\documentclass[12pt]{article}

\usepackage{jones}
\usepackage{makecell}

\DeclareMathOperator{\route}{route}
\DeclareMathOperator{\cycles}{cycles}

\DeclareMathOperator{\gloop}{gg}
\newcommand{\Id}{\text{Id}}
\newcommand{\symbool}{\text{Sym}^V_{bool}}
\newcommand{\PM}{\mathcal{PM}}
\newcommand{\M}{\mathcal{M}}
\newcommand{\falling}[1]{\underline{\underline{#1}}}
\newcommand{\rising}[1]{\overline{\overline{#1}}}
\newcommand{\lrarrow}{\leftrightarrow}
\newcommand{\notlrarrow}{\not\lrarrow}

\title{Almost-Orthogonal Bases for Inner Product Polynomials}

\author{Chris Jones\thanks{University of Chicago. {\tt csj@uchicago.edu.} Supported in part by NSF grant CCF-2008920.
}\and Aaron Potechin\thanks{University of Chicago. {\tt potechin@uchicago.edu.} Supported in part by NSF grant CCF-2008920.}}

\date{\vspace{-5ex}}

\begin{document}

\maketitle
\begin{abstract}
    In this paper, we consider low-degree polynomials of inner products between a collection of random vectors. We give an almost orthogonal basis for this vector space of polynomials when the random vectors are Gaussian, spherical, or Boolean. In all three cases, our basis admits an interesting combinatorial description based on the topology of the underlying graph of inner products. 
    
    We also analyze the expected value of the product of two polynomials in our basis.
    In all three cases, we show that this expected value can be expressed in terms of collections of matchings on the underlying graph of inner products. In the Gaussian and Boolean cases, we show that this expected value is always non-negative. In the spherical case, we show that this expected value can be negative but we conjecture that if the underlying graph of inner products is planar then this expected value will always be non-negative.
\end{abstract}

\section{Introduction}
When we have a collection of random variables, it is often extremely useful to find a basis of polynomials in the random variables which is orthonormal under the natural inner product $\ip{f}{g} \defeq \E[f\cdot g]$. Some important examples are as follows:
\begin{enumerate}
    \item If $x$ is a random point of the Boolean hypercube $\{-1,1\}^n$ then the multilinear monomials $\{\prod_{i \in S}{x_i}: S \subseteq [n]\}$ are an orthonormal basis.
    \item When we have a single Gaussian variable $x \sim \calN(0,1)$, the Hermite polynomials (with the correct normalization) are an orthonormal basis. When $x$ is an $n$-dimensional vector with Gaussian coordinates (i.e. $x \sim \calN(0,\Id_n)$), the multivariate Hermite polynomials form an orthonormal basis.
    \item When $x \in \mathbb{R}^n$ is a random unit vector (i.e. $x \unif S^{n-1}$), spherical harmonics give an orthonormal basis.
\end{enumerate}

In this paper, we consider polynomials of inner products between a collection of random vectors. More precisely,
%Fix a countably infinite set $V$ 
fix a finite set of vertices $V$ and $n \in \N$ and consider drawing i.i.d. random $n$-dimensional
%standard $n$-dimensional Gaussian 
vectors $d_u$ for each $u \in V$. We will work in three settings: when the $n$-dimensional vector $d_u$ is a standard Gaussian, a uniform unit vector, and a uniform Boolean vector. We consider polynomials in the variables $d_{u,i}$ with real coefficients which have degree less than $n$ and are orthogonally invariant i.e. unchanged if the $\{d_u\}$ are simultaneously replaced by $\{Td_u\}$ for any orthogonal matrix $T$. Any such orthogonally invariant polynomial will also be expressible\footnote{When $d_u$ is a Boolean vector, we instead require that the polynomials are invariant under permutations of $[n]$ and changing the signs of coordinates (i.e. automorphisms of the Boolean hypercube). In this setting, in addition to inner products, we also have $k$-wise inner products for all even $k > 2$. For more details, see \cref{sec:boolean-case}.} in terms of the inner product variables $x_{uv} \defeq \ip{d_u}{d_v}$.

A natural spanning set for the space of orthogonally invariant polynomials is the set of monomials $\prod_{u, v \in V} x_{uv}^{k_{uv}}$ where each $k_{uv} \in \N$. Equivalently, there is one monomial for each undirected multigraph on $V$ (with self-loops allowed in the Gaussian case): 
for the monomial $\prod_{u, v \in V} x_{uv}^{k_{uv}}$ we take the graph where there are $k_{uv}$ multi-edges from $u$ to $v$. We denote this monomial by $m_G$ where $G$ is the underlying graph.

However, the monomials $m_G$ are not orthogonal. For example, one can check that in the Gaussian case, the graph shown in \cref{fig:four-cycle} has
\[ \E[x_{12}x_{23}x_{34}x_{14}] = \E_{d_1,d_2,d_3,d_4 \sim \mathcal{N}(0, \Id_n)}[\ip{d_1}{d_2}\ip{d_2}{d_3}\ip{d_3}{d_4}\ip{d_1}{d_4} ] = n.\]
\begin{figure}[h]
    \centering
    {\begin{tikzpicture}[scale=0.15]
\tikzstyle{every node}+=[inner sep=0pt]
\draw [black] (36.7,-34.3) circle (3);
\draw (36.7,-34.3) node {$4$};
\draw [black] (48.8,-34.3) circle (3);
\draw (48.8,-34.3) node {$3$};
\draw [black] (36.7,-20.8) circle (3);
\draw (36.7,-20.8) node {$1$};
\draw [black] (48.8,-20.8) circle (3);
\draw (48.8,-20.8) node {$2$};
\draw [black] (36.7,-23.8) -- (36.7,-31.3);
\draw [black] (39.7,-34.3) -- (45.8,-34.3);
\draw [black] (48.8,-31.3) -- (48.8,-23.8);
\draw [black] (45.8,-20.8) -- (39.7,-20.8);
%\draw [black] (38.8,-23.1) -- (46.8,-32.06);
%\draw [black] (38.7,-32.07) -- (46.8,-23.03);
\end{tikzpicture}}
    \caption{}
    \label{fig:four-cycle}
\end{figure}
Our goal in this paper is to orthogonalize the $m_G$ into a basis of polynomials $p_G$. As it turns out, the basis $p_G$ which we will obtain is not quite orthogonal, but it is very close. In particular, we will have that $\ip{p_G}{p_H} = 0$ unless $V(G) = V(H)$ and $G$ and $H$ have the same degree at every vertex. In addition, even when $G \neq H$ and $\ip{p_G}{p_H} \neq 0$, $\ip{p_G}{p_H}$ will be small (see Lemma \ref{lem:approximate-inversion}).

While the $p_G$ basis is not quite orthogonal, it exhibits some surprisingly beautiful combinatorics based on the underlying graph $G$.
Even computing $\E[m_G]$, one can already see a connection to the topology of
the graph $G$. In the Gaussian case, the magnitude of 
$\E[m_G]$ is $n^k$ where $k$ is the maximum number of cycles that $E(G)$ can be partitioned into (and is 0 if $G$ has a vertex with odd degree) and analogous results hold for the spherical and Boolean cases (see \cref{lem:monom-expectation}, \cref{lem:sphere-monom-expectation}, and \cref{lem:boolean-expectation}).
A theme of this paper is that quantities involving the $m_G$ and $p_G$ may not have clean
exact formulas, but their magnitudes in $n$ are determined by combinatorial and topological
properties of $G$.

\subsection{Outline}

In the remainder of the introduction, we give more overview on the $p_G$ in general.
In \cref{sec:gaussian-case} we specialize to the Gaussian case $d_u \sim \calN(0, \Id_n)$.
In the Gaussian case, the calculations work out cleanly once one has
the right definitions.
%``matching collections''.
In \cref{sec:spherical-case} we continue to the spherical case. Here the calculations become more involved, and we investigate a conjecture relating the spherical case to planar graphs.
In \cref{sec:boolean-case} we investigate the Boolean case $d_u \unif \{-1,+1\}^n$ which combines aspects of the Gaussian and spherical cases. In \cref{sec:inversion-formula} we give a ``Fourier inversion'' lemma for potential applications. \cref{sec:appendix} lists some small polynomials. For the purpose of 
gaining intuition about the family $p_G$, it may be helpful to carry around a few small examples from the tables in~\cref{sec:appendix} and see how the results and proofs apply to these polynomials.

\subsection*{Acknowledgements}

We would like to thank Goutham Rajendran for discussions and many comments on this work.
We also thank Mrinalkanti Ghosh, Fernando Granha Jeronimo, and Madhur Tulsiani for early
discussions on the polynomials in the context of Sum-of-Squares.

\subsection{Constructing the polynomials}

Given any inner product on polynomials, we can automatically construct an orthonormal basis of polynomials by using the Gram-Schmidt process.
%We construct our basis by applying Gram-Schmidt to the $\{m_G\}$, producing a family of orthogonal polynomials $\{p_G\}$. 
However, to run Gram-Schmidt, it is necessary to choose an order. A natural order for polynomials is by degree, though within each degree it is not clear how the polynomials should be ordered. We skirt this issue by only orthogonalizing a monomial against polynomials with lower degree\footnote{Degree of a polynomial in this paper always
refers to total degree.}.
%\footnote{A further issue to running Gram-Schmidt is that we have countably many variables, and therefore infinitely many polynomials of each degree. We will see that $m_G$ is already orthogonal to all but finitely many of the smaller-degree $p_H$.}
The resulting polynomials we produce are ``mostly orthogonal'', with $\E[p_G \cdot p_H]$ possibly nonzero for polynomials of the same degree (in fact, they will be orthogonal unless $G$ and $H$ have the same degree on every vertex).
We call this the \textit{degree-orthogonal Gram-Schmidt process}.

\begin{definition}
    A polynomial family $\{p_I\}_{I \in \calI}$ is degree-orthogonal (with respect to $\calD$) if $\E_{d \sim \calD}[p_I(d) p_J(d)]~=~0$ whenever
    $\deg(p_I) \neq \deg(p_J)$.
\end{definition}

The degree-orthogonal Gram-Schmidt process outputs the unique monic degree-orthogonal basis. 
\begin{fact}[(Uniqueness of Gram-Schmidt orthogonalization)]
\label{lem:unique-gs}
Let $\{m_I\}_{I \in \calI}$ be the set of monomials of degree at most $\tau$ in
a set of variables $\nu$ and let $\calD$ be a distribution on $\R^{\nu}$
such that $\{m_I\}_{I \in \calI}$ are linearly independent as functions on the support of $\calD$. There is a unique set of monic polynomials $\{p_I\}_{I \in \calI}$ such that 
\begin{enumerate}[(i)]
\item The unique monomial of maximum degree in $p_I$ is $m_I$,
\item The family $p_I$ is degree-orthogonal with respect to $\calD$.
\end{enumerate}
Furthermore, the $p_I$ are linearly independent and span the same space as the $m_I$.
\end{fact}
\begin{proof}
Condition (i) says that $p_I$ lies in the space $\linspan(\{m_I\} \cup \{m_J \; : \; \deg(m_J) < \deg(m_I)\})$. Condition (ii) says that $p_I$ is orthogonal to the latter subspace of codimension 1, and therefore $p_I$ is determined since it's monic.
% 
% Linear independence of the $m_I$ as functions on $\supp(\calD)$ implies linear independence of the $p_I$.
\end{proof}
\begin{remark}
    Our monomials $\{m_G\}$ are not linearly independent when the degree is too high. In this case, $\{p_G\}$ will be a spanning set rather than a basis.
\end{remark}
% is lin indpt on supp(D) equivalent to E[fg] being nondegenerate?

However, Gram-Schmidt certainly does not guarantee any nice description of the resulting
polynomials. It turns out that the $p_G$ also have closed-form combinatorial descriptions and we now give one such description.
%(the choice of ``mostly orthogonal'' was made so that the $p_G$ have a clean description, but it's also possible that there are other nice families of orthogonal polynomials that can be obtained by performing more orthogonalization within each degree).
%We now give a general and combinatorial description of the $p_G$. 
However, calculations are still a pain using this description.
In the next sections we will give alternate combinatorial formulas for the $p_G$ based on %``matching collections'' 
collections of matchings that allow for calculations, and also highlight
the connection between the $p_G$ and the topology of the graph $G$.

Let $d \sim \calD^{\otimes V}$ for some distribution $\calD$ on $\R^n$, which we will later take
to be either Gaussian, uniformly spherical, or uniformly Boolean.
We want to find an orthogonal polynomial
basis for $\Aut(\calD)$\nobreakdash-invariant functions. Let $\{\chi_{\alpha}: \alpha \in \N^n\}$ be the monic polynomial family on $\R^n$ which is degree-orthogonal under $\calD$ (this is the orthogonal basis for entries of a single vector, e.g. the multivariate Hermite polynomials in the Gaussian case).
We assume that we have a set $\{m_I\}_{I \in \calI}$ of homogeneous polynomials in the $d_{u,i}$ which form a (not necessarily orthogonal) basis for the $\Aut(\calD)$-invariant functions up to a certain degree. For example, this can be the inner product functions $ \prod_{u,v \in V}\ip{d_u}{d_v}^{k_{uv}}$ in the Gaussian and spherical cases. Construct $\{p_I\}_{I \in \calI}$ by applying to $m_I$ the map (extending by linearity),
\[ \prod_{u \in V} d_u^{\alpha_u} \mapsto \prod_{u \in V} \chi_{\alpha_u}(d_u).\]
In words, each monomial is replaced by the $\calD$-orthogonal polynomial with that leading monomial. %(Up to a small caveat explained below) 
As shown by the following proposition, the $p_I$ are monic, degree-orthogonal, and have the same degree as the corresponding $m_I$, so if the $p_I$ are $\Aut(\calD)$-invariant then the $p_I$ are a monic degree-orthogonal basis for the $\Aut(\calD)$-invariant functions, and hence equal the output of the Gram-Schmidt process on $m_I$.
% For example, in the Gaussian case with $n=2$,
% \[ \ip{d_1}{d_2} = d_{1,1}d_{2,1} + d_{\]
\begin{proposition}
    \label{prop:generic-construction}
    The polynomials $\{p_I\}_{I \in \calI}$ are monic, satisfy $\deg(p_I) = \deg(m_I)$, and
    are degree-orthogonal.
    \begin{proof}
        The degree is preserved by the map sending $d_u^\alpha \mapsto \chi_\alpha(d_u)$ and the leading coefficient is 1 since the $\chi_\alpha$ are monic.
        Suppose that $p_I, p_J$ have distinct degrees; then so do $m_I, m_J$. 
        For each term in the expression
        $p_I p_J$,
        because $m_I, m_J$ are homogeneous and have different degree, 
        there must be $u \in V$ such that
        the degree in $d_u$ differs between $p_I, p_J$.
        Because of degree-orthogonality of the $\chi_\alpha$,
        the expectation over $d_u$ is zero.
    \end{proof}
\end{proposition}
However, it's not clear that the new polynomials $p_I$ have the desired $\Aut(\calD)$ symmetry
without more assumptions on $\calD$. For our settings we will check that this is indeed the case.

\subsection{Related work}

Some of the combinatorics of the monomials $m_G$ is captured by the \emph{circuit partition polynomial}~\cite{Bollobas02CircuitPartitionPolynomial} (see also 
the \emph{Martin polynomial}~\cite{MartinThesis, EM98MartinPolynomial}) which is the univariate
generating function for circuit partitions of $G$:
\[r_G(x) = \sum_{k \geq 0} r_k(G) x^k \]
where $r_k(G)$ is the number of ways to split the edges of $G$ into exactly $k$ circuits.
${r_G(n) = \E[m_G]}$ for the Gaussian distribution, as we show in \cref{lem:monom-expectation}.
This formula
was also
computed by Moore and Russell~\cite{MR10CircuitPartitions}, who also
prove the spherical case, \cref{lem:sphere-monom-expectation}.

Although Gram-Schmidt works well for univariate polynomials, in general finding an \emph{explicit} orthogonal basis of polynomials for a given space is a difficult task.
Examples include polynomials on the unit ball and simplex~\cite{DaiXu-SphericalPolysBook} or a
slice of the hypercube~\cite{Filmus16SliceBasis}.
Occasionally it is simpler to find a degree-orthogonal family, as we do here.
For example, ``the'' spherical harmonics (as originally given by Laplace in $n=3$ dimensions,
see Chapter 4 of~\cite{DunklXuBookOrthogonalPolynomials} for general $n$)
are an orthogonal basis for functions on the sphere.
However, it is easier to use the ``Maxwell representation'', which is only degree-orthogonal, as we do in \cref{sec:spherical-case}.
% The Grassmann graph~\cite{DKKMS18NonexpandingGrassmannSets}, high-dimensional expanders~\cite{DDFH18HDXFourierAnalysis}, and all association schemes are further examples
% where degree-orthogonality is easier to obtain.

To the best of our knowledge, the $p_G$ have not been explored before.
We now compare the $p_G$ with several similar families of polynomials.

When $G$ equals $k$ multiedges between two vertices 1 and 2, $p_G$
generalizes a univariate orthogonal polynomial family evaluated on $\ip{d_1}{d_2}$. For the spherical case
this is the Gegenbauer polynomials. For the Boolean case, this is the Kravchuk polynomials (after an affine shift).
For the Gaussian case, $p_G$ also depends on $\norm{d_1}$ and $\norm{d_2}$, but
evaluated on $\ip{d_1}{d_1} = \ip{d_2}{d_2} = n$ this is the (probabilist's) Hermite polynomials. 

For a collection of jointly Gaussian random variables $X_i$, 
the \textit{Wick product} gives a monic orthogonal polynomial family under the expectation inner product~\cite{JansonBookGaussianHilbertSpaces}. In our set-up, there are two differences with 
the Wick product. First, the variables $x_{uv}$ are not themselves Gaussian; they are individually distributed as $\sqrt{X} \cdot Y$ where $X$ is a chi-squared random variable with $n$ degrees of freedom and $Y$ is an independent standard Gaussian. This however could be fixed by using bipartite graphs $G$ and sampling $d_u$ as a Gaussian vector on one bipartition and as a spherical vector on the other. The second and more important difference is that even if this change is made, the $x_{uv}$ are individually Gaussian but not jointly Gaussian. The graph structure of $G$ enforces nontrivial correlations. For example, in the four-cycle given earlier in Figure~\ref{fig:four-cycle}, each edge is mean-zero and each pair of edge variables is uncorrelated, and so if the variables were jointly Gaussian then they would be independent and mean-zero. However, $\E[x_{12}x_{23}x_{34}x_{14}] > 0$.

The \textit{matching polynomial} of a graph $G$ is the univariate generating function for the number of matchings in $G$. Despite both families generalizing e.g. the Hermite polynomials, the matching polynomials and $p_G$ seem incomparable.

For a permutation group $G \leq S_k$, one defines the \textit{cycle polynomial}~\cite{CameronSemeraro18CyclePolynomial}
% \footnote{This is a specialization of the more general \textit{cycle index polynomial}.}
\[\displaystyle\sum_{g \in G} x^{\text{number of cycles in }g}.\]
Though this is similar in appearance to some calculations in this paper, there is not a clear group $G$ associated with the matching structures that we consider.

\subsection{Applying the \texorpdfstring{$p_G$}{pG} basis}

We end the introduction by describing how the $p_G$ basis may be applied. 
The $p_G$ basis behaves like a Fourier basis for orthogonally invariant functions of a collection of vectors $d = \{d_u\}$. While other bases may be simpler, the $p_G$ basis is specialized to orthogonally invariant functions and it exhibits nontrivial combinatorial cancellations which would be hard to spot and explain in other bases, and which might be intrinsic to some problems. %not to mention using the $p_G$ is fun and enjoyable.
We expect that the $p_G$ will be most useful for applications where we work with large $n$ and relatively low-degree moments of $\calD$, such as analyzing the sum of squares hierarchy or the trace power method at low degrees.

We encountered the $p_G$ basis in the course of the work~\cite{GJJPR20SherringtonKirkpatrickPlantedAffinePlanes}, in which a superset of the current authors prove lower bounds against the sum of squares hierarchy for the Sherrington-Kirkpatrick problem. Technically, this work constructs a matrix $\calM$ which is a function of a collection of random Gaussian vectors $\{d_u\}$; the entries of $\calM$ are naturally expressed (via ``pseudocalibration'') in terms of an orthogonal polynomial basis evaluated on the $d_u$. 
Ultimately, we ended up using the standard Hermite basis as this was sufficient for our purposes, though we also considered using the $p_G$ basis.

\section{Polynomial Basis for the Gaussian Setting}
\label{sec:gaussian-case}

In this section we investigate the family $\{p_G\}$ when $d_u \sim \calN(0, \Id_n)$ 
i.i.d. The graph $G$ is a multigraph on $V$, possibly with self-loops. 
We will develop a combinatorial understanding of the polynomials through ``routings'' (\cref{def:routing})
and use it to give formulas for the inner product (\cref{lem:gaussian-inner-product}) and variance (\cref{cor:gaussian-variance}).

Note that the $m_G$ are not completely linearly independent. For example,
if $n = 1$, then $m_G$ is determined by its degrees on each vertex.
Despite this, the low-degree monomials are linearly independent.
\begin{lemma}
    The set of $m_G$ for $\abs{E(G)} \leq n$ is linearly independent.
\end{lemma}
\begin{proof}
    Suppose that $\sum_{G: \abs{E(G)} \leq n} c_G m_G = 0$; we show $c_G = 0$.
    Each inner product $\ip{d_u}{d_v}$ can be expanded as $\sum_{i=1}^n d_{u,i}d_{v,i}$. In this way, each edge gets a label from 1 to $n$. Expanding $m_G$,
    \[ m_G = \sum_{\sigma: E(G) \to [n]} \prod_{\{u,v\} \in E(G)} d_{u, \sigma(\{u,v\})}d_{v, \sigma(\{u,v\})}.\]
    Since $\abs{E(G)} \leq n$, one monomial that appears in $m_G$ will have $\sigma$ assign a distinct label to each edge.
    We claim that this monomial appears in the sum with coefficient $c_G$: because the edge labels are distinct, we can recover the graph $G$ from the monomial.
    Therefore $c_G = 0$.
\end{proof}
For low-degree polynomials the $p_G$ will therefore be a basis.

The polynomials $p_G$ admit several nice combinatorial descriptions based on the graph $G$. To see why something combinatorially nice might be expected to happen, there is a combinatorially-flavored method for computing $\E[m_G]$ via Isserlis' theorem (also known as Wick's lemma).
\begin{lemma}[(Isserlis' theorem)]
Fix vectors $d_1, \dots, d_{2k} \in \R^n$. Then for $v$ a standard $n$-dimensional Gaussian random variable,
\[\E_{v} [\ip{v}{d_1}\cdots \ip{v}{d_{2k}}] = \displaystyle\sum_{\substack{\text{perfect matchings}\\ M \text{ on }[2k]}} \prod_{(u,v) \in {M}} \ip{d_u}{d_v}.\]
Observe also that the expectation is zero when there are an odd number of inner products.
\end{lemma}

We will need the following minor generalization.
\begin{lemma}
    For fixed $d_1, \dots, d_{2k} \in \R^n$ and $v \sim \calN(0,\Id_n),$
    \[\E_{v} [\ip{v}{v}^{2l}\ip{v}{d_1} \cdots \ip{v}{d_{2k}}] =  n(n+2)\cdots (n+2l-2)  \sum_{\substack{\text{perfect matchings}\\ M \text{ on }[2k]}} \prod_{(u,v) \in {M}} \ip{d_u}{d_v}.\]
\begin{proof}
    See~\cref{app:isserlis-theorem}, \cref{thm:isserlis-proof}.
\end{proof}
\end{lemma}
The generalization can be iterated to compute $\E[m_G]$ for a given graph $G$.
We take the expectation
over the vectors $d_u$ one at a time, and each application reduces our expression to a sum over graphs that no longer involve $u$.

To capture the combinatorics of the $p_G$, 
we look at matchings of the edge endpoints incident to a given vertex. More specifically we use a collection $M$ of (partial or perfect) matchings on
incident edges, one for each vertex.

\begin{definition}
Let $\mathcal{PM}(G)$ be the set of all perfect matching collections of the edges incident to each vertex of $G$. Each element of $\mathcal{PM}(G)$ specifies $\abs{V(G)}$ perfect matchings, and the perfect matching for vertex $v$ is on $\deg(v)$ elements. 

Let $\calM(G)$ denote the set of all partial or perfect matching collections of the edges incident to each vertex of $G$.
\end{definition}

\begin{definition}
For $M \in \cal{M}(G)$, define the routed graph $\route(M)$ to be the graph obtained by connecting up edge endpoints that are matched at each vertex $v$. Closed cycles are deleted, and paths are replaced by a single edge between the final path endpoints. 
\end{definition}
\begin{definition}
For $M \in \cal{M}(G)$, define $\cycles(M)$ to be the number of closed cycles formed by routing.
\end{definition}

We give an example in \cref{fig:routing-example}. The graph on the left has 5
vertices, and 10 edges denoted by solid lines. 
The edges are partially matched up at each vertex using the
dashed edges. The right side shows the result of routing.
$\cycles(M) = 1$ and one closed cycle, the triangle, was deleted.

\begin{figure}[h]
    \centering
    \includegraphics[width=0.25\textwidth]{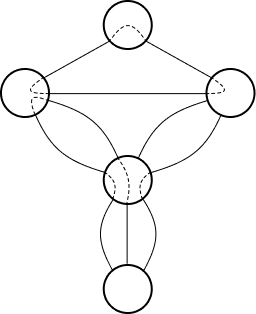}\hspace{1in}
    \includegraphics[width=0.25\textwidth]{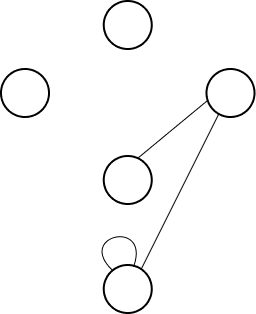}
    \caption{Left: Unrouted graph with dashed edges denoting the partial matching collection. Right: Result of routing.}
    \label{fig:routing-example}
\end{figure}

Using these definitions, we have the following formula for the expectation $\E[m_G]$,
\begin{lemma}
\label{lem:monom-expectation}
$\E[m_G] = 0$ if some vertex in $G$ has odd degree. Otherwise,
\[\E [m_G]= \displaystyle\sum_{M \in \PM(G)} n^{\cycles(M)}.\]
\begin{proof}
Expanding $m_G$ and grouping by vertex,
\[ m_G = \displaystyle\sum_{\sigma : E \to [n]} \prod_{u \in V, i \in [n]} d_{u, i}^{\#\{e\ni u \;:\; \sigma(e) = i\}}. \]
Taking expectations, the $d_{u,i}$ are independent Gaussians. If one of the vertices has odd degree, one of the labels $i$ will necessarily occur an odd number of times at that vertex and the overall expectation will be zero. Otherwise, $\E \left[Z^{2k} \right]= (2k-1)!!$ for $Z \sim \mathcal{N}(0,1)$. The expression $(2k-1)!!$ counts the number of perfect matchings of $2k$ elements; in this case when computing $\E \left[d_{u, i}^{\#\{e\ni u \;:\; \sigma(e) = i\}}\right]$ these should be thought of as summing 1 for each perfect matching of the edges incident to $u$ which are labeled $i$. In summary, each $\sigma$ sums over a subset of $\mathcal{PM}$.

Now fix a given collection of perfect matchings $M \in\mathcal{PM}$; which $\sigma$ contribute to it? We require that, at each vertex, every pair of endpoints matched in $M$ are assigned the same label. Therefore, in any cycle formed by $\route(M)$, the labeling $\sigma$ must assign all edges of the cycle the same label. These labels can be any number from $[n]$, and disjoint cycles don't affect each other. Therefore there are $n^{\cycles(M)}$ such $\sigma$.
\end{proof}
\end{lemma}

\begin{corollary}
The magnitude of $\E[m_G]$ is $n^k$ where $k$ is the maximum number of cycles into which $E(G)$ can be partitioned (note that this is NP-hard to compute from $G$).
\end{corollary}

We now give several alternate definitions of the polynomials $p_G$.
\begin{definition}[(Hermite sum definition)]
\label{def:hermites}
Define $p_G$ by
\[p_G = \displaystyle\sum_{\sigma : E(G) \to [n]} \prod_{u \in V, i \in [n]} h_{\abs{\{e \ni u \; : \; \sigma(e) = i\}}}(d_{u,i}). \]
\end{definition}
Note that in this definition we consider a self-loop at $u$ labeled $i$ to contribute 2 to $\abs{\{e \ni u \; : \; \sigma(e) = i\}}.$
\begin{definition}[(Routing definition)]
\label{def:routing}
Define $p_G$ by
\[p_G = \displaystyle\sum_{M \in \mathcal{M}(G)} m_{\route(M)} \cdot n^{\cycles(M)} \cdot (-1)^{\abs{M}}. \]
\end{definition}

A given graph $K$ can appear as $\route(M)$ for several different matchings $M$ (even with different numbers of cycles). This gives rise to interesting and nontrivial coefficients on the monomials $m_K$.

\begin{definition}[(Generic construction from~\cref{prop:generic-construction})]
\label{def:truncation}
Consider the Hermite expansion of $m_G$ in the variables $d_{u,i}$, and let $p_G$ be the truncation to the top level i.e. keep only those Hermite coefficients $c_\alpha h_\alpha$ with $\abs{\alpha} = 2\abs{E(G)}$. 

Since $m_G$ is homogeneous as a function of the $d_{u, i}$ and each monomial appears with coefficient 1, this amounts to taking each monomial $d^\alpha$ and replacing it by $h_\alpha(d)$.
\end{definition}

\begin{lemma}
The three definitions above are equivalent.
\end{lemma}
\begin{proof}
After checking that the leading monomial of $p_G$ in \cref{def:hermites} is $m_G$, it is clear that \cref{def:hermites} and \cref{def:truncation} are equivalent.

We argue \cref{def:hermites} and \cref{def:routing} agree. The coefficient of $x^{k-2i}$ in the Hermite polynomial $h_k(x)$ can be interpreted as the number of matchings of $2i$ objects out of $k$ total (there is also an alternating sign). In this way $h_k$ is a generating function for all partial matchings on $[k]$. Looking at $h_{\abs{\{e \ni u \; : \; \sigma(e) = i\}}}(d_{u,i})$, we interpret this (ignoring the sign for now) as summing over all partial matchings on the incident edges with a given label $i$; any matched edges are given a factor of 1 while the unmatched edges are given $d_{u,i}$.

Now look at the view from a given collection $M$ of partial matchings, one per vertex. Which $\sigma$ contribute? Along any closed cycle in $\route(M)$, the labels $\sigma$ must be all the same, and if this is the case the contribution is 1. Along any path, the labels assigned by $\sigma$ must also be the same, say $i$. The multiplicative contribution of any interior vertices is 1, but the contribution of the two endpoint vertices $u$ and $v$ is $d_{u,i}$ and $d_{v,i}$. When all valid $\sigma$ are summed over, we obtain a factor of $n$ for each cycle, and the inner product between the endpoints of each path.

The $(-1)^{\abs{M}}$ factor comes from the signings of the Hermite coefficients.
\end{proof}
\begin{example}
Let $G$ be the graph with vertices $V(G) = \{u,v_1,v_2,v_3\}$ and edges $E(G) = \{\{u,v_1\}, \{u,v_2\}, \{u,v_3\}\}$. We have that 
\begin{align*}
m_G &= \ip{d_u}{d_{v_1}}\ip{d_u}{d_{v_2}}\ip{d_u}{d_{v_3}}\\
&=\sum_{\text{distinct } i,j,k \in [n]}{d_{u,i}d_{u,j}d_{u,k}d_{{v_1},i}d_{{v_2},j}d_{{v_3},k}} + \sum_{i \neq j \in [n]}{d_{u,i}^{2}d_{u,j}d_{{v_1},i}d_{{v_2},i}d_{{v_3},j}}  \\
&+ \sum_{i \neq j \in [n]}{d_{u,i}^{2}d_{u,j}d_{{v_1},i}d_{{v_2},j}d_{{v_3},i}}+ \sum_{i \neq j \in [n]}{d_{u,i}^{2}d_{u,j}d_{{v_1},j}d_{{v_2},i}d_{{v_3},i}} + \sum_{i \in [n]}{d_{u,i}^{3}d_{{v_1},i}d_{{v_2},i}d_{{v_3},i}}.
\end{align*}
Replacing the monomial $d_{u,i}^{2}$ with the corresponding Hermite polynomial $d_{u,i}^{2} - 1$ and replacing the monomial $d_{u,i}^{3}$ with the corresponding Hermite polynomial $d_{u,i}^{3} - 3d_{u,i}$, we have that 
\begin{align*}
p_G &=\sum_{\text{distinct } i,j,k \in [n]}{d_{u,i}d_{u,j}d_{u,k}d_{{v_1},i}d_{{v_2},j}d_{{v_3},k}} + \sum_{i \neq j \in [n]}{(d_{u,i}^{2} - 1)d_{u,j}d_{{v_1},i}d_{{v_2},i}d_{{v_3},j}} \\
&+ \sum_{i \neq j \in [n]}{(d_{u,i}^{2} - 1)d_{u,j}d_{{v_1},i}d_{{v_2},j}d_{{v_3},i}} 
+ \sum_{i \neq j \in [n]}{(d_{u,i}^{2} - 1)d_{u,j}d_{{v_1},j}d_{{v_2},i}d_{{v_3},i}} \\
&+ \sum_{i \in [n]}{(d_{u,i}^{3} - 3d_{u,i})d_{{v_1},i}d_{{v_2},i}d_{{v_3},i}}
\end{align*}
The term $-d_{u,j}d_{{v_1},i}d_{{v_2},i}d_{{v_3},j}$ and one of the $3$ terms $-d_{u,i}d_{{v_1},i}d_{{v_2},i}d_{{v_3},i}$ correspond to the collection of matchings $M$ where $v_1$ is matched to $v_2$ at $u$ (and all other matchings are trivial). Summing these terms over all $i \neq j \in [n]$ gives  $-\ip{d_{v_1}}{d_{v_2}}\ip{d_u}{d_{v_3}}$.

Similarly, the term $-d_{u,j}d_{{v_1},i}d_{{v_2},j}d_{{v_3},i}$ and one of the $3$ terms $-d_{u,i}d_{{v_1},i}d_{{v_2},i}d_{{v_3},i}$ correspond to the collection of matchings $M$ where $v_1$ is matched to $v_3$ at $u$ (and all other matchings are trivial). Summing these terms over all $i \neq j \in [n]$ gives  $-\ip{d_{v_1}}{d_{v_3}}\ip{d_u}{d_{v_2}}$.

Finally, the term $-d_{u,j}d_{{v_1},j}d_{{v_2},i}d_{{v_3},i}$ and one of the $3$ terms $-d_{u,i}d_{{v_1},i}d_{{v_2},i}d_{{v_3},i}$ correspond to the collection of matchings $M$ where $v_2$ is matched to $v_3$ at $u$ (and all other matchings are trivial). Summing these terms over all $i \neq j \in [n]$ gives  $-\ip{d_{v_2}}{d_{v_3}}\ip{d_u}{d_{v_1}}$.

Putting everything together,
\[
p_G = \ip{d_u}{d_{v_1}}\ip{d_u}{d_{v_2}}\ip{d_u}{d_{v_3}} -\ip{d_{v_1}}{d_{v_2}}\ip{d_u}{d_{v_3}} -\ip{d_{v_1}}{d_{v_3}}\ip{d_u}{d_{v_2}} -\ip{d_{v_2}}{d_{v_3}}\ip{d_u}{d_{v_1}}.
\]
\end{example}

As a consequence of the routing definition we have
\begin{lemma}
The polynomials $p_G$ are orthogonally invariant.
\end{lemma}

\begin{corollary}
Definitions~\ref{def:hermites},~\ref{def:routing},~\ref{def:truncation} are 
equal to the degree-orthogonal Gram-Schmidt process on $m_G$.
    \begin{proof}
    From~\cref{prop:generic-construction}, the polynomials defined above are 
    degree-orthogonal and monic. The previous lemma shows that they are orthogonally invariant. Therefore they match the result of Gram-Schmidt by~\cref{lem:unique-gs}.
    \end{proof}
\end{corollary}

In fact, the proof of~\cref{prop:generic-construction} shows that $p_G$ have a stronger ``ultra-orthogonality'' property.
If $G$ and $H$ have different degree at $u$, then only taking the expectation over $d_u$ already results in the zero polynomial.
\begin{lemma}
Let $G$ and $H$ be two multigraphs. If $\deg_G(u) \neq \deg_H(u)$ for some $u \in V$, then
\[\E_{d_u \sim \calN(0,\Id_n)}[p_G \cdot p_H] = 0.\]
\end{lemma}

We now derive an explicit formula for the inner product and variance of $p_G$. For two graphs $G,H$ on $V$, we define $G \cup H$ to be the disjoint union of the edges (the edge multiplicity in $G \cup H$ is the sum of the multiplicities in $G$ and $H$). 

\begin{definition}
For two multigraphs, write $G \lrarrow H$ if $\deg_G(u) = \deg_H(u)$ for all $u \in V$.
\end{definition}

\begin{definition}
Let $\mathcal{PM}(G, H) \subseteq \PM(G \cup H)$ be perfect matching collections 
such that at each vertex $v$, the matching goes between edges incident to $v$ in 
$G$ and edges incident to $v$ in $H$. Note that if $G \notlrarrow H$, then 
$\mathcal{PM}(G, H)$ is empty.
\end{definition}

\begin{lemma}
\label{lem:gaussian-inner-product}
Let $G$ and $H$ be two multigraphs. 
\[ \E[p_G \cdot p_H] = \displaystyle\sum_{M \in \mathcal{PM}(G, H)} n^{\cycles(M)} \]
\end{lemma}
\begin{remark}
Determining the maximum number of cycles in $M \in \PM(G,H)$ is NP-hard via a slight modification of~\cite{Holyer81EdgePartitioning}. 
This remains true if we restrict the cycles to be simple.
\end{remark}
\begin{proof}
Use the routing definition of $p_H$,
\[p_G \cdot p_H = \displaystyle\sum_{\substack{M_1 \in \mathcal{M}(G), \\ M_2 \in \calM(H)}} m_{\route_G(M_1)} \cdot m_{\route_H(M_2)} \cdot n^{\cycles_G(M_1) + \cycles_H(M_2)} \cdot (-1)^{\abs{M_1} + \abs{M_2}} .\]
Taking expectations, by Lemma~\ref{lem:monom-expectation} we sum over all perfect matchings of $\route_G(M_1) \cup \route_H(M_2)$ which ``complete'' the partial matchings $M_1$ and $M_2$, when viewed as a matching on the graph $G \cup H$. The power of $n$ is the number of cycles in the completed matching. The net effect is to sum over all perfect matching collections in $G \cup H$,
\[\E [p_Gp_H] = \displaystyle\sum_{M \in  \mathcal{PM}(G \cup H)} n^{\cycles_{G \cup H}(M)} \sum_{\text{pick some $M$-matched pairs to be in } M_1 \text{ or }M_2} (-1)^{\abs{M_1} + \abs{M_2}}. \]
The inner summation often cancels to zero. In the graph $G \cup H$, each edge-vertex incidence comes from either $G$ or $H$. We can only add an $M$-matched pair to $M_1$ if both matched edge-vertex incidences come from $G$; similarly only matched pairs which are both in $H$ can be picked for $M_2$. If there are any such pairs, the inner summation is automatically zero.

The remaining terms are those $M$ in which, at every vertex, the perfect matching is a perfect matching between incoming edges in $G$ and those in $H$ -- that is, matching collections in $\mathcal{PM}(G,H)$. For these terms, the inner summation is trivially 1, which finishes the proof.
\end{proof}

\begin{corollary}
\label{cor:gaussian-variance}
$n^{\abs{E(G)}} \leq \E[p_G^2] \leq \abs{E(G)}^{2\abs{E(G)}} \cdot n^{\abs{E(G)}}$.
\begin{proof}
Using the result of Lemma~\ref{lem:gaussian-inner-product}, we claim $\max_{M \in \mathcal{PM}(G, G)} \cycles(M) = \abs{E(G)}$. On the one hand, $\abs{E(G)}$ is achievable by matching each edge with its duplicate to create 2-cycles. On the other hand, every cycle in $\route(M)$ needs at least two edges (there can be no self-loops as matchings with self-loops are not in $\mathcal{PM}(G, G)$). This shows $n^{\abs{E(G)}} \leq \E p_G^2 \leq \abs{\mathcal{PM}(G, G)} \cdot n^{\abs{E(G)}}$.

$\mathcal{PM}(G, G)$ consists of choosing a perfect matching at each vertex between two sets of size $\deg(v)$.
\[\abs{\mathcal{PM}(G, G)} = \prod_{v \in V} \deg(v)! \leq \abs{E(G)}^{2\abs{E(G)}}.\]
\end{proof}
\end{corollary}

\section{Polynomial Basis for the Spherical Setting}
\label{sec:spherical-case}

Let $S^{n-1} = \{x \in \R^n : \norm{x}_2 = 1\}$. With the $d_u$ drawn uniformly
and independently from $S^{n-1}$ instead of the Gaussian distribution, for each multigraph $G$
with no self-loops (reflecting the fact that $\ip{d_u}{d_u} = 1$) we
construct a polynomial $p_G$. 
We again construct the polynomials in terms of routings (\cref{def:routing-sphere}) and study
the inner product (\cref{sec:spherical-inner-product}) and variance (\cref{cor:spherical-variance}).
For the most part, the proofs in this section mirror their counterparts in the previous section, with the notable
exception of the inner product formula, which exhibits surprising mathematical depth.

Let $\alpha \in \N^n$ be a multi-index and $\abs{\alpha} \defeq \sum_{i=1}^n \alpha_i$. We will need the Maxwell representation of harmonic polynomials~\cite[Theorem 1.1.9]{DaiXu-SphericalPolysBook}. Concretely, let the spherical harmonic $s_\alpha:~S^{n-1}~\to~\R$ be (the restriction to $S^{n-1}$ of the function on $\R^n$)
\[s_\alpha(x) = \norm{x}^{2\abs{\alpha} + n-2} \pd{}{x^\alpha}\norm{x}^{-n+2} \]
and then scaled to be monic. An alternate method to write down $s_\alpha$ is to first write down the Hermite polynomial $\prod_{i=1}^n h_{\alpha_i}(x_i)$ and then multiply each non-leading monomial of total degree $\abs{\alpha} - 2k$ by approximately\footnote{Multiplying the monomials by exactly $n^{-k}$ creates polynomials orthogonal under the distribution $\calN(0, \Id_n/n)$, which is similar to the unit sphere.} $n^{-k}$. More precisely, we let $x^{\underline{\underline{k}}}$ be notation for the ``fall-by-2'' falling factorial, \[x^{\underline{\underline{k}}} \defeq x(x-2)(x-4)\cdots(x-2k+2) , \]
and let $x^{\underline{\underline{-k}}} \defeq 1/x^{\underline{\underline{k}}}$. We also define $x^{\overline{\overline{k}}}$ likewise for rise-by-2. Then:
\begin{fact}
\label{fact:hermite-to-sphere}
To form $s_\alpha$ from $h_\alpha$, multiply monomials with degree $\abs{\alpha} - 2k$ by $(n+2\abs{\alpha}-4)^{\underline{\underline{-k}}}$.
\end{fact}

We will need the moments of the uniform distribution on the sphere (using the notation introduced above):
\begin{fact}
\[\E_{x \unif S^{n-1}} [x^\alpha] = 
\left\{\begin{array}{lr} n^{\overline{\overline{-\abs{\alpha}/2}}} \cdot \E_{Z \sim \calN(0,\Id_n)} [Z^\alpha]  & \text{If }\alpha_i \text{ even for all } i\\
0 & \text{Otherwise}
\end{array}\right.\]
\end{fact}

These spherical harmonics are degree-orthogonal (as functions of a 
single vector):
\begin{fact}
If $\abs{\alpha} \neq \abs{\beta}$, then $\displaystyle\E_{x \unif S^{n-1}} [s_\alpha(x) s_\beta(x)] = 0$.
\end{fact}

We remark that $\{s_\alpha : \alpha \in \N^n\}$ is not completely linearly independent as functions on $S^{n-1}$ because of the identity $\ip{v}{v} = 1$:
\begin{fact}
For each $k$, the set $\{s_\alpha : \abs{\alpha} \leq k, \alpha_n = 0 \text{ or }1\}$ is a basis for the set of degree-$(\leq k)$ polynomial functions on $S^{n-1}$.
The same holds for the monomials $x^\alpha$.
\end{fact}

The monomials $m_G$ are defined as before for each multigraph on $V$ without self-loops.
They are not completely
linearly independent as functions on $(S^{n-1})^V.$
We restrict ourselves to the set of low-degree functions, which are linearly independent.
\begin{lemma}
    The set of $m_G$ with $\abs{E(G)} \leq n-1$ is linearly independent
    as functions on $(S^{n-1})^V$.
\end{lemma}
\begin{proof}
    Suppose $\sum_{G : \abs{E(G)} \leq n}c_G m_G = 0$ where $c_G$ are not all zero, and let 
    $G$ be a nonzero graph with maximum number of edges.
    Expanding $m_G$,
    \[ m_G = \sum_{\sigma : E(G) \to [n]} \prod_{\{u,v\} \in E(G)} d_{u,\sigma(\{u,v\})} d_{v, \sigma(\{u,v\})}.\]
    Letting $\sigma$ be an injective assignment of labels from $[n-1]$ (which exists
    because $\abs{E(G)} \leq n-1$), we claim that the corresponding monomial, which we call the ``special monomial'',
    is uncancelled and appears with coefficient $c_G$.

    First, the relations $\ip{d_u}{d_u} = 1$ mean that polynomials do not have
    a unique representation as functions on $(S^{n-1})^V$. We amend this
    by using the relations to reduce the degree of variable $d_{u,n}$ to 0 or 1 for each vertex $u$, replacing $d_{u,n}^2 = 1 - \sum_{i=1}^{n-1}d_{u,i}^2$.
    Nothing needs to be done for the special monomial.
    
    After performing the replacement, the special monomial still does not arise from
    any other graphs. This is because the reduction step must either lower the degree,
    or introduce a variable with degree 2, whereas the special monomial is multilinear
    and was chosen to have maximum degree. Therefore, the special monomial has coefficient $c_G$, which
    is nonzero, a contradiction.
\end{proof}

There is a spherical Isserlis theorem which gives a recursive method to compute $\E[m_G]$.
\begin{lemma}[(Spherical Isserlis theorem)]
Fix vectors $d_1, \dots, d_{2k} \in \R^n$. Then for $v \unif S^{n-1}$,
\[\E_{v} [\ip{v}{d_1}\cdots \ip{v}{d_{2k}}] = n^{\overline{\overline{-k}}} \displaystyle\sum_{\substack{\text{perfect matchings}\\ \mathcal{M} \text{ on }[2k]}} \prod_{(u,v) \in \mathcal{M}} \ip{d_u}{d_v}.\]
Observe also that the expectation is zero when there are an odd number of inner products.
\end{lemma}
\begin{proof}
This follows from the standard Isserlis theorem. 
Let $Q \sim\chi^2(n)$ be a chi-square random variable with $n$ 
degrees of freedom, independent from $v$. Then
\[ \E_{v, Q} \left[\ip{\sqrt{Q}v}{d_1}\cdots \ip{\sqrt{Q}v}{d_{2k}}\right] = \E_{Z\sim \calN(0,\Id_n)} [\ip{Z}{d_1}\cdots \ip{Z}{d_{2k}}].\]
Factoring out $Q^k$, the left-hand side is
\[ \E_{v, Q} [\ip{\sqrt{Q}v}{d_1}\cdots \ip{\sqrt{Q}v}{d_{2k}}] = \E_Q[Q^k] \cdot \E_{v} [\ip{v}{d_1}\cdots \ip{v}{d_{2k}}] 
= n^{\rising{k}} \E_{v} [\ip{v}{d_1}\cdots \ip{v}{d_{2k}}].\]
By the Gaussian Isserlis theorem, the right-hand side equals
\[\E_{Z\sim \calN(0,\Id_n)} [\ip{Z}{d_1}\cdots \ip{Z}{d_{2k}}] = 
\displaystyle\sum_{\substack{\text{perfect matchings}\\ \mathcal{M} \text{ on }[2k]}} \prod_{(u,v) \in \mathcal{M}} \ip{d_u}{d_v}.\]
Dividing by $n^{\rising{k}}$ proves the claim.
\end{proof}

We also have explicit formulas based on matching collections,
\begin{lemma}
\label{lem:sphere-monom-expectation}
$\E[m_G] = 0$ if there is a vertex of odd degree, otherwise, 
\[\E[m_G] = \prod_{v \in V} n^{\overline{\overline{-\deg(v)/2}}} \displaystyle\sum_{M \in \mathcal{PM}(G)} n^{\cycles(M)}.\]
\begin{proof}
The proof goes through exactly as in the Gaussian case, 
but plug in the spherical moments which contribute the rising factorial terms.
\end{proof}
\end{lemma}

\begin{remark}
\label{rmk:sphere-monom-loops}
\cref{lem:sphere-monom-expectation} is still valid if $G$ has self-loops.
\end{remark}

We now give three definitions of the orthogonal polynomials for the spherical case which are analogous to Definitions \ref{def:hermites}, \ref{def:routing}, and \ref{def:truncation} for the Gaussian case:
\begin{definition}[(Spherical harmonic sum definition)]
\label{def:harmonics}
Define $p_G$ by
\[p_G =  \displaystyle\sum_{\sigma : E(G) \to [n]} \prod_{u \in V} s_{\text{histogram of }\{\sigma(e): e \ni u\}}(d_u).\]
\end{definition}

\begin{definition}[(Routing definition)]
\label{def:routing-sphere}
Define $p_G$ by
\[p_G =  \displaystyle\sum_{M \in \mathcal{M}(G)} m_{\route(M)} \cdot n^{\cycles(M)} \cdot (-1)^{\abs{M}} \cdot \prod_{v \in V} (n+2\deg(v) - 4)^{\underline{\underline{-\abs{M_v}}}} \]
where $M_v$ is the partial matching of incident edges at $v$.
\end{definition}

\begin{definition}[(Generic construction from~\cref{prop:generic-construction})]
\label{def:truncate-sphere}
To construct $p_G$, expand the function $m_G$ in the basis of spherical harmonics as a function of $d_{u,i}$ then truncate to the top-level coefficients of degree $2\abs{E(G)}$.
\end{definition}

\begin{lemma}
The three definitions above are equivalent.
\begin{proof}
Definitions~\ref{def:harmonics} and \ref{def:truncate-sphere} agree once we check that the leading monomial in \cref{def:harmonics} is $m_G$.

Definitions~\ref{def:harmonics} and \ref{def:routing-sphere} agree as a consequence of equality between \cref{def:hermites} and \cref{def:routing} in the Gaussian case by the following argument. For reference, we recall the two equal formulas for $p_G$ in the Gaussian case,
\begin{equation} \label{eq:def1}
p_G = \displaystyle\sum_{\sigma : E(G) \to [n]} \prod_{u \in V, i \in [n]} h_{\abs{\{e \ni u \; : \; \sigma(e) = i\}}}(d_{u,i})
\end{equation}
\begin{equation} \label{eq:def2}
p_G = \displaystyle\sum_{M \in \mathcal{M}(G)} m_{\route(M)} \cdot n^{\cycles(M)} \cdot (-1)^{\abs{M}}.
\end{equation}
For each fixed $\delta \in V^\N$, let us restrict to only the monomials in the variables $d_{u,i}$ with total degree $\delta(u)$ on the variables $\{d_{u,i} : i \in [n]\}$. We clearly still have equality between~\cref{eq:def1} and~\cref{eq:def2} after making this restriction. The equality still holds if we multiply both sides by an appropriate function of $n$; we choose this function of $n$ to be the product that appears on the right side of \cref{def:routing-sphere}, which only depends on $\delta$. This clearly converts \cref{eq:def2} into \cref{def:routing-sphere}. Due to the choice of function, it also turns \cref{eq:def1} into \cref{def:harmonics} because of the conversion between $h_\alpha$ and $s_\alpha$ in~\cref{fact:hermite-to-sphere}.
\end{proof}
\end{lemma}

The following properties follow directly as they did in the Gaussian case:
\begin{lemma}
The polynomials $p_G$ are orthogonally invariant.
\end{lemma}
\begin{lemma}
\label{lem:sphere-ultraorthogonal}
Let $G$ and $H$ be two multigraphs. If $\deg_G(u) \neq \deg_H(u)$ for some $u\in V$,
then $\E_{d_u \unif S^{n-1}}[p_G \cdot p_H] = 0$.
\end{lemma}
\begin{lemma}
Definitions~\ref{def:harmonics},~\ref{def:routing-sphere},~\ref{def:truncate-sphere} are equal to the output of the degree-orthogonal Gram-Schmidt process on $m_G$.
\end{lemma}

\subsection{Inner product}\label{sec:spherical-inner-product}

Unfortunately, we do not have a clean formula for the inner product of two spherical polynomials.
Compared to the Gaussian case, there are several complications.
First, in the spherical case some polynomials with $G \lrarrow H$ are orthogonal (whereas in the Gaussian case $p_G$ and $p_H$ are orthogonal iff $G \notlrarrow H$, via \cref{lem:gaussian-inner-product}). For example, the following two polynomials are orthogonal:
    \[ p_G = x_{12}x_{13}x_{45} - \frac{x_{23}x_{45}}{n}, \qquad p_H = x_{14}x_{15}x_{23} - \frac{x_{45}x_{23}}{n}.\]
This shows that extra cancellations occur in the spherical case.
Second, when $n$ is small some of the $p_G$ are degenerate. For example, if $G$ is a triangle and $n = 2$ then $p_G = 0$.
Third, even in the asymptotic regime of large $n$ and constant-size graphs, 
in the spherical case the inner product may be negative
(whereas in the Gaussian case the inner product is always non-negative).
For example, this occurs if $E(G), E(H)$ partition the edges of $K_5$, with 
the inner 5-cycle in $G$ and the outer 5-cycle in $H$, as in \cref{fig:k5}.

\begin{figure}[h]
    \centering
    \includegraphics[width=0.3\textwidth]{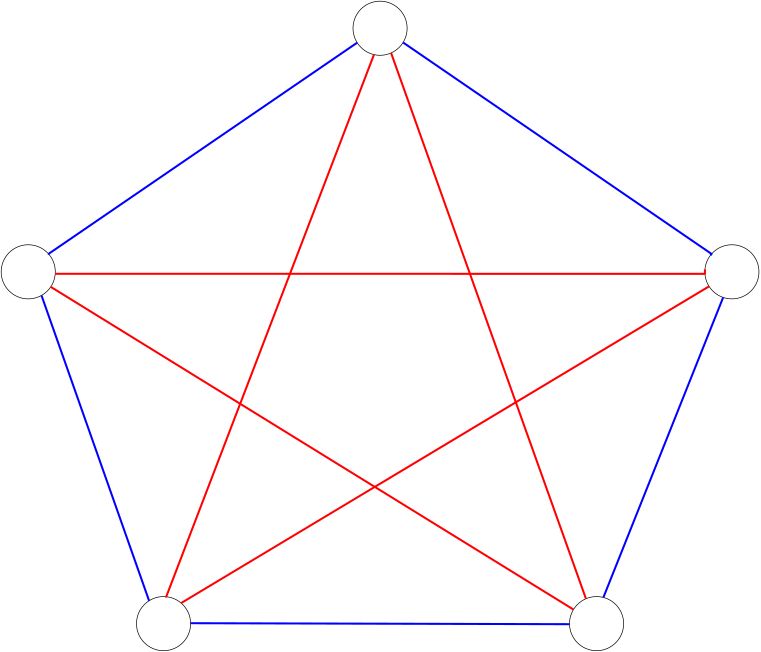}
    \caption{}
    \label{fig:k5}
\end{figure}

In this case it can be computed~(see \cref{app:k5})
that
\[\E[p_G \cdot p_H] = \frac{-8(n-1)(n-2)(n-4)}{n^8(n+2)^4}.\]
Interestingly, we conjecture that negative inner product can only occur
if the graph $G \cup H$ is nonplanar.

To attack these complications, we first give a general expression for the inner product.
We use it to upper bound the magnitude of the inner product, showing
that it's no larger than the Gaussian case, up to normalization (\cref{lem:generic-ip-ubound}).
We then study some situations when cancellations occur in an effort to determine
the exact magnitude in $n$ of the inner product.
Due to the inherent difficulties, this section is a bit technical.

The proof strategy we use is to consider the contribution $c_M$ from each matching collection $M \in \PM(G\cup H)$ and 
then isolate cancellations that occur between these terms (similarly to
how the inner product was computed in the Gaussian case, \cref{lem:gaussian-inner-product}).
\begin{definition}
For some $M \in \mathcal{PM}(G \cup H)$, define a \textit{$G$-pair} as a pair of matched endpoints in $M$ where both come from $G$. An $H$-pair and a $(G,H)$-pair are defined analogously. 
\end{definition}
Let $g_M(v)$ denote the number of $G$-pairs at vertex $v$ and $g_M$ denote the total
number of $G$-pairs.

If $G \not \leftrightarrow H$ then $\E[p_G p_H] = 0$ %by ultra-orthogonality 
by \cref{lem:sphere-ultraorthogonal},
so we may assume $G \leftrightarrow H$.

\begin{lemma}
\label{lem:cM}
Let $G,H$ be
arbitrary multigraphs such that $G \leftrightarrow H$. 
Let $d(v) = \deg_G(v) = \deg_H(v)$.
Then
\[\E[p_G \cdot p_H] = \prod_{v \in V} n^{\rising{-d(v)}}\displaystyle\sum_{M \in \PM(G \cup H)} c_M \]
where the coefficients $c_M$ are
\[ c_M = n^{\cycles(M)}\prod_{v \in V} \frac{(-2)^{\falling{g_M(v)}}}{(n+2d(v)-4)^{\falling{g_M(v)}}}.\]
\end{lemma}
\begin{proof}
By orthogonality, $\E[p_G \cdot p_H] = \E[p_G \cdot m_H]$.
Using the routing definition,
\begin{multline*}
p_G \cdot m_H = \displaystyle\sum_{M \in \mathcal{M}(G)} m_{\route_G(M)} \cdot m_{H} \cdot n^{\cycles_G(M)} \cdot (-1)^{\abs{M}} \cdot \prod_{v \in V} (n+2d(v) - 4)^{\underline{\underline{-\abs{M_v}}}}.
% = \displaystyle\sum_{\substack{M_1 \in \mathcal{M}(G), \\ M_2 \in \calM(H)}} m_{\route_G(M_1)} \cdot m_{\route_H(M_2)} \cdot n^{\cycles_G(M_1) + \cycles_H(M_2)} \cdot (-1)^{\abs{M_1} + \abs{M_2}}\\ \cdot \prod_{v \in V(G \cup H)} n^{-\abs{(M_1)_v} -\abs{(M_2)_v}} \cdot (1 + O(1/n))
\end{multline*}
Taking expectations\footnote{We should not remove the self-loops in $\route_G(M)$ which is permitted by~\cref{rmk:sphere-monom-loops}.} using~\cref{lem:sphere-monom-expectation}, we expand $\E[m_{\route_G(M)} \cdot m_{H}]$ into a sum over all completions $C$ of the partial matching $M$ (on the graph $G \cup H$).
As in the Gaussian case, we collect terms based on the overall matching $M \cup C \in \mathcal{PM}(G\cup H)$.
Performing the grouping of terms, we have
\begin{align*}
\E p_Gm_H &= \displaystyle\sum_{M \in  \mathcal{PM}(G \cup H)} n^{\cycles_{G \cup H}(M)} \sum_{S \subseteq \text{$G$-pairs}} (-1)^{\abs{S}}
\cdot \prod_{v \in V}  (n+2d(v) - 4)^{\falling{-\abs{S_v}}} \cdot  n^{\rising{-d(v) + \abs{S_v}}}\\
&= \displaystyle\sum_{M \in  \mathcal{PM}(G \cup H)} n^{\cycles(M)} \prod_{v \in V}\sum_{S_v \subseteq \text{$G$-pairs at }v} (-1)^{\abs{S_v}}
(n+2d(v) - 4)^{\falling{-\abs{S_v}}} \cdot  n^{\rising{-d(v) + \abs{S_v}}}\\
&= \displaystyle\sum_{M \in  \mathcal{PM}(G \cup H)} n^{\cycles(M)} \prod_{v \in V}\sum_{k=0}^{g_M(v)} \binom{g_M(v)}{k}(-1)^{k}
(n+2d(v) - 4)^{\falling{-k}} \cdot  n^{\rising{-d(v) + k}}.
\end{align*}
The inner summation (with $v$ fixed) is
\begin{align*}
 &\sum_{k=0}^{g_M(v)} \binom{g_M(v)}{k}(-1)^{k}
(n+2d(v) - 4)^{\falling{-k}} \cdot  n^{\rising{-d(v) + k}}\\
&= n^{\rising{-d(v)}}\sum_{k=0}^{g_M(v)} \binom{g_M(v)}{k}(-1)^{k}
(n+2d(v) - 4)^{\falling{-k}} \cdot  (n+2d(v)-2)^{\falling{k}}\\
&= \frac{n^{\rising{-d(v)}}}{(n+2d(v)-4)^{\falling{g_M(v)}}}\sum_{k=0}^{g_M(v)} \binom{g_M(v)}{k}(-1)^{k}
(n+2d(v) - 2g_M(v) -2)^{\rising{g_M(v) - k}} \cdot  (n+2d(v)-2)^{\falling{k}}\\
&= \frac{n^{\rising{-d(v)}}}{(n+2d(v)-4)^{\falling{g_M(v)}}}\sum_{k=0}^{g_M(v)}\binom{g_M(v)}{k}
(n+2d(v) - 2g_M(v)-2)^{\rising{g_M(v) - k}} \cdot  (-n-2d(v)+2)^{\rising{k}}.
\end{align*}
    Using the umbral formula $(x+y)^{\rising{m}} = \displaystyle\sum_{k=0}^m \binom{m}{k}x^{\rising{k}} y^{\rising{m-k}}$~\cite{RomanBook},
\begin{align*}
    &= \frac{n^{\rising{-d(v)}}}{(n+2d(v)-4)^{\falling{g_M(v)}}}(-2g_M(v))^{\rising{g_M(v)}}\\
    &= \frac{n^{\rising{-d(v)}}}{(n+2d(v)-4)^{\falling{g_M(v)}}}(-2)^{\falling{g_M(v)}}.
\end{align*}

\end{proof}

\begin{corollary}
\label{cor:spherical-variance}
For $G$ such that $\abs{E(G)} \leq o(\log n /\log\log n)$, 
\[\E [p_G^2] = n^{-\abs{E(G)} + o(1)}.\]
\end{corollary}
\begin{proof}
    We have 
    \[\prod_{v \in V} n^{\rising{-d(v)}} = \prod_{v \in V}n^{-d(v) + o(1)} = n^{-2\abs{E(G)} + o(1)}.\]
    The magnitude of the coefficient $c_M$ is $n^{\cycles(M) - g_M}$.
    Since $G$ has no self-loops, the max number of cycles for $M \in \PM(G \cup G)$ is 
    $\abs{E(G)}$, therefore $M$ has
    the largest magnitude of $n$ if and only if $M$ pairs each edge with
    a parallel edge from the other copy of the graph.
    For these $M$, $c_M = n^{\abs{E(G)}}$. There is at least one such $M$ and
    possibly up to $\abs{PM(G, H)}.$ Under the size assumption on $G$,
    $\abs{PM(G,H)} = n^{o(1)}$ and therefore non-dominant terms are negligible,
    \[ \E[p_G^2] = n^{-2\abs{E(G)} + \abs{E(G)} + o(1)} = n^{-\abs{E(G)} + o(1)}.\]
\end{proof}

Up to the normalization factor of $\prod_{v \in V}n^{\rising{-d(v)}}$, the
inner product is bounded by the same formula from the Gaussian case.
\begin{corollary}
    \label{lem:generic-ip-ubound}
    Let $G$ and $H$ be two multigraphs such that $G \leftrightarrow H$ with degrees $d(v)$, and $\abs{E(G)}, \abs{E(H)} \leq o(\log n/\log \log n)$. Then
    \[\abs{\E[p_G \cdot p_H]} \leq \prod_{v \in V} n^{\rising{-d(v)}}\displaystyle\sum_{M \in \mathcal{PM}(G, H)} n^{\cycles(M) + o(1)}.\]
\end{corollary}
\begin{proof}
From~\cref{lem:cM},
\begin{align*}\E[p_G \cdot p_H] &= \prod_{v \in V} n^{\rising{-d(v)}} \displaystyle\sum_{M \in \PM(G \cup H)} c_M\\
&= \prod_{v \in V} n^{\rising{-d(v)}} \displaystyle\sum_{M \in \PM(G \cup H)} n^{\cycles(M)}\prod_{v \in V} \frac{(-2)^{\falling{g_M(v)}}}{(n+2d(v)-4)^{\falling{g_M(v)}}}.
\end{align*}
If $M$ has both a $G$-pair and an $H$-pair at $v$, observe how the magnitude
of $c_M$ changes if we re-match them into two $(G,H)$-pairs to get a new
matching $M'$.
$g_M(v)$ goes down by 1. $\cycles(M)$ may increase by 1, decrease by 1, 
or stay the same. Therefore the magnitude of $c_{M'}$ is at least as large as $c_M$.
Iterating this, the dominant terms are $M \in \PM(G, H)$, and the size assumption means
they are dominant up to a $n^{o(1)}$ factor.
\end{proof}

There are often significantly more cancellations than the Gaussian case. 
We conjecture that the magnitude for planar graphs $G \cup H$ is given by the \textit{simple} matchings $M\in\PM(G, H)$. 
\begin{definition}
For a multigraph $G$ and $M \in \PM(G)$, we say that $M$ is $v$-simple if $v$ is visited at most once in each cycle induced by $M$. We say that $M$ is simple if every cycle is simple.
\end{definition}
\begin{conjecture}
\label{conj:planar-inner-product}
Let $G$ and $H$ be two loopless multigraphs such that $G \cup H$ is planar, and $\abs{E(G)}, \abs{E(H)}~\leq~o(\log n / \log \log n)$.
Then
\[ \E[p_G \cdot p_H] = \frac{1}{n^{\abs{E(G)} + \abs{E(H)}}}\cdot\left(\displaystyle\sum_{\text{simple }M \in \mathcal{PM}(G, H)} n^{\cycles(M)}\right)\cdot (1 \pm o(1)). \]
If there are no simple $M \in \PM(G, H)$, then the expectation is zero.
\end{conjecture}

$K_5$ is a counterexample to an extension of the conjecture to non-planar graphs.
$\E[p_G p_H]~<~0$ for $G$ and $H$ equal to two 5-cycles despite that:
\begin{proposition}
    Decomposing $K_5$ into two 5-cycles $G$ and $H$,
    there is no simple $M \in \PM(G, H)$.
\end{proposition}
\begin{proof}
    The cycles created by $M \in \PM(G, H)$ are necessarily even-length since
    they alternate between $G$ and $H$ edges. They can't be length-2 since $E(G) \cap E(H) = \emptyset.$
    Therefore there must be two cycles of lengths 4 and 6, or one cycle of length 10, 
    but a length-6 or 10 cycle is not simple.
\end{proof}
There are other examples with $K_5$ minors with negative inner product.
Taking $G,H$ to be the red and blue edges in \cref{fig:k5-minor1},
\[\E[p_G p_H] = \frac{-16(n-1)(n-2)(n-4)}{n^{11}(n+2)^5}.\]
Taking $G, H$ to be the red and blue edges in \cref{fig:k5-minor2},
\[ \E[p_G p_H] = \frac{-16(n-1)(n-2)^2(n-4)}{n^{11}(n+2)^6}.\]
\begin{figure}[h]
    \centering
\begin{minipage}{.5\textwidth}
  \centering
  {\includegraphics[width=0.7\textwidth]{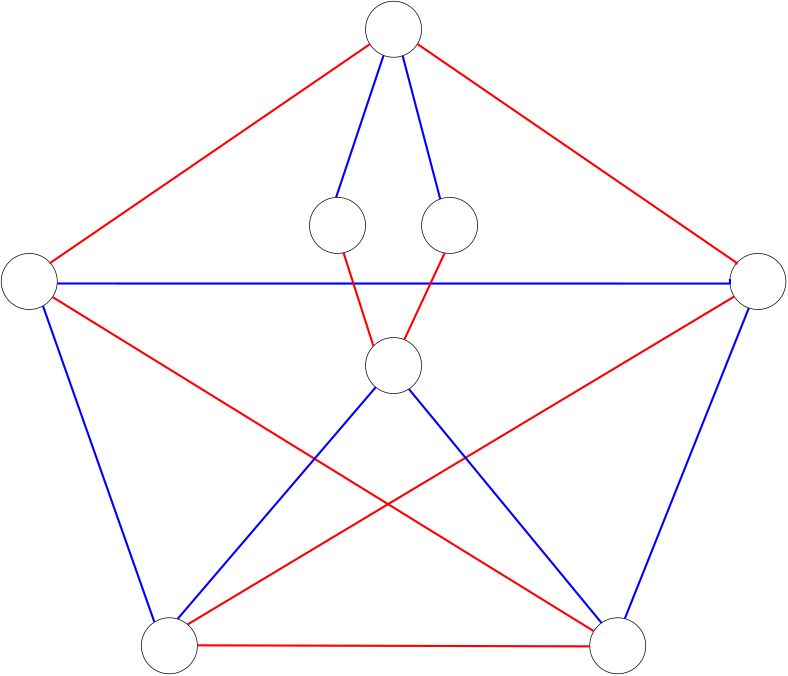}}
  \caption{}
  \label{fig:k5-minor1}
\end{minipage}%
\begin{minipage}{.5\textwidth}
  \centering 
  {\includegraphics[width=0.7\textwidth]{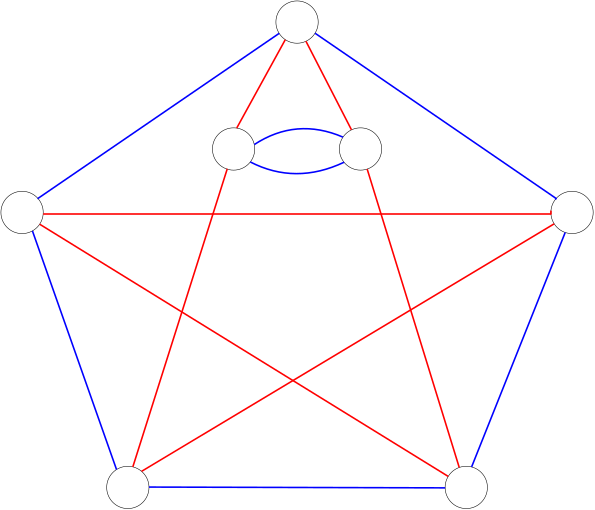}}
  \caption{}
  \label{fig:k5-minor2}
\end{minipage}%
\end{figure}

We don't know of any similar examples based on $K_{3,3}$. 
To lend support to the conjecture, we consider an approach that almost works, and
show some intuition for why the failure of the approach is related to planarity
(or is at least topological in nature).

Observe that fixing a matching collection off of a vertex $v$ induces a matching of the edges incident to $v$.
To wit, if you leave $v$ along edge $e$, and follow the fixed matching around
outside of $v$, you will eventually return to $v$ along some edge.
If the induced matching at $v$ has a $G$-pair
then we claim that summing over matchings at $v$ produces zero.
\begin{lemma}
\label{lem:cancellation}
Let $M'$ be a perfect matching collection for all vertices except $v$.
If there is a $G$-pair in the induced matching at $v$, then
\[ \sum_{M: M\text{ extends }M'\text{ at }v} c_M = 0.\]
\end{lemma}

\begin{proof}
Abbreviate the sum in the statement as $\sum_{M \psdgeq M'}$.
The $c_M$ factor out a term for vertices that are not $v$,
\[\sum_{M \psdgeq M'} c_M = \prod_{\substack{w \in V,\\w \neq v}} \frac{(-2)^{\falling{g_{M'}(w)}}}{(n+2d(w)-4)^{\falling{g_M(w)}}}\sum_{M \psdgeq M'} n^{\cycles(M)} \frac{(-2)^{\falling{g_M(v)}}}{(n+2d(v)-4)^{\falling{g_M(v)}}}. \]
We will argue that the latter sum is zero.

Let $e_1$ and $e_2$ be two edges which form a $G$-pair induced by the matchings $M'$ outside of $v$. Consider the following map from matchings of the edges incident to $v$ where $e_1$ is not matched to $e_2$ to matchings where $e_1$ is matched to $e_2$. If $e_1$ is matched to $e_i$ and $e_2$ is matched to $e_j$ then we match $e_1$ and $e_2$ and match $e_i$ and $e_j$.

To invert this mapping, given a matching where $e_1$ is matched with $e_2$, we need to know which of the $d(v) - 1$ other matched pairs $e_i$ and $e_j$ to swap with and we need to know whether  to match $e_1$ with $e_i$ and $e_2$ with $e_j$ or $e_1$ with $e_j$ and $e_2$ with $e_i$.

Consider a given matching where $e_1$ is matched with $e_2$. Letting $c+1$ be the number of cycles in the matching and $k+1$ be the number of 
$G$-pairs at $v$, this matching gives a value of 
\[
n^{c+1}\cdot \frac{(-2)^{\falling{k}}}{{(n+2d(v)-4)^{\falling{k}}}}\cdot \frac{-2(k+1)}{n+2d(v)-2k-4}.
\]
We now show that this term cancels with the terms for all of the matchings where $e_1$ is not matched to $e_2$ which map to this matching. 
Observe an important property of the induced matching: for all re-matchings
of $e_1$ and $e_2$ with $e_i,e_j$, the number of cycles decreases by exactly 1.
For the $2(k+1)$ matchings where $e_1$ and $e_2$ are mixed with an $H$-pair, each such matching gives a value of 
\[
n^c \cdot \frac{(-2)^{\falling{k}}}{(n+2d(v)-4)^{\falling{k}}}.
\]
For the $2d(v) - 2k - 4$ matchings where $e_1$ and $e_2$ are mixed with another $G$-pair 
or are matched with a $(G,H)$-pair, each such matching gives a value of 
\[
n^c \cdot \frac{(-2)^{\falling{k}}}{{(n+2d(v)-4)^{\falling{k}}}}\cdot\frac{-2(k+1)}{n+2d(v)-2k-4}.
\]
Adding these terms together and dividing by $n^c \cdot \frac{(-2)^{\falling{k}}}{(n+2d(v)-4)^{\falling{k}}}$, we obtain
\[
\frac{-2(k+1)n}{n+2d(v)-2k-4} + 2(k+1) - \frac{2(k+1)(2d(v)-2k-4)}{n+2d(v)-2k-4} = 0.
\]
\end{proof}
The next corollary explains some cancellations.
\begin{corollary}\label{cancellationcorollary}
    If $G \cup H$ has a cut vertex $v$ such that a component $C$ of $(G \cup H) \setminus v$
    has an unequal number of $G$ edges and $H$ edges incident to $v$,
    then $\E[p_G p_H] = 0$.
\end{corollary}
\begin{proof}
    Fixing any perfect matching collection on $C$, this necessarily induces either an $H$-pair or a $G$-pair
    at $v$. By the previous lemma, summing over the matchings at $v$ yields zero.
\end{proof}
% \cnote{Also give a probabilistic proof!}

Even when we cannot apply Corollary \ref{cancellationcorollary}, Lemma \ref{lem:cancellation} can still be very useful in computing $\E[p_G p_H]$.
\begin{example}
\label{ex:two-4-cycles}
    Consider the graphs $G$ and $H$ depicted in \cref{fig:two-four-cycles} where $V(G) = V(H) = \{1,2,3,4\}$, $E(G) = \{\{1,2\}, \{1,2\}, \{3,4\}, \{3,4\}\}$, and $E(H) = \{\{2,3\}, \{2,3\}, \{4,1\}, \{4,1\}\}$.
    
    We can compute $\E[p_G p_H]$ as follows. Consider vertex $1$ and the edges incident to it. We partition the collections of matchings based on how these edges are connected to each other in the remainder of the graph. We then sum over the possible matchings at vertex $1$.
    
    Let $e_1$ and $e_2$ be the two copies of $\{1,2\}$ and let $e_3$ and $e_4$ be the two copies of $\{1,4\}$. 
    \begin{enumerate}
        \item If there is a path from $e_1$ to $e_2$ and a path from $e_3$ to $e_4$ (outside of vertex $1$) then by Lemma \ref{lem:cancellation}, everything cancels at vertex $1$.
        \item If there is a path from $e_1$ to $e_3$ and a path from $e_2$ to $e_4$ (outside of vertex $1$) then summing over the matchings at vertex $1$ gives
        \[
        \frac{n^2}{n(n+2)} + \frac{n}{n(n+2)} - \frac{2n}{n^2(n+2)} = \frac{n^2 + n - 2}{n(n+2)} = \frac{n-1}{n}
        \]
        To see this, note that the first term corresponds to the matching $\{e_1,e_3\}, \{e_2,e_4\}$ at vertex $1$ as this gives two cycles and gives a factor of $\frac{1}{n(n+2)}$ for vertex $1$. The second term corresponds to the matching $\{e_1,e_4\}, \{e_2,e_3\}$ at vertex $1$ as this gives one cycle and gives a factor of $\frac{1}{n(n+2)}$ for vertex $1$. The third term corresponds to the matching $\{e_1,e_2\}, \{e_3,e_4\}$ at vertex $1$ as this gives one cycle and gives a factor of $\frac{-2}{n^2(n+2)}$ for vertex $1$.
        
        Note that in order to have these paths, there must be $G$-$H$ matchings at the other $3$ vertices and there are $4$ ways to do this and route $e_1$ to $e_3$
        and $e_2$ to $e_4$. When there are $G$-$H$ matchings at the other $3$ vertices, each of these vertices gives a factor of $\frac{1}{n(n+2)}$
        \item The case when there is a path from $e_1$ to $e_4$ and a path from $e_2$ to $e_3$ behaves in the same way as the previous case.
    \end{enumerate}
    Adding everything together, 
    \[
    \E[p_G p_H] = 2 \cdot 4 \cdot \left(\frac{1}{n(n+2)}\right)^{3} \cdot \frac{n-1}{n} = \frac{8(n-1)}{n^4(n+2)^3}
    \]
\end{example}
\begin{figure}[h]
    \centering
\begin{tikzpicture}[scale=0.15]
\tikzstyle{every node}+=[inner sep=0pt]
\draw [black] (22.9,-17.6) circle (3);
\draw (22.9,-17.6) node {$1$};
\draw [black] (43.1,-17.6) circle (3);
\draw (43.1,-17.6) node {$4$};
\draw [black] (43.1,-35.1) circle (3);
\draw (43.1,-35.1) node {$3$};
\draw [black] (22.9,-35.1) circle (3);
\draw (22.9,-35.1) node {$2$};
\draw [black] (21.651,-32.376) arc (-160.18984:-199.81016:17.782);
\draw (20.1,-26.35) node [left] {$G$};
\draw [black] (23.955,-20.406) arc (16.50389:-16.50389:20.925);
\draw (25.32,-26.35) node [right] {$G$};
\draw [black] (25.692,-34.008) arc (107.77772:72.22228:23.935);
\draw (33,-32.36) node [above] {$H$};
\draw [black] (40.268,-36.086) arc (-74.04667:-105.95333:26.445);
\draw (33,-37.6) node [below] {$H$};
\draw [black] (25.689,-16.499) arc (107.92481:72.07519:23.756);
\draw (33,-14.85) node [above] {$H$};
\draw [black] (40.252,-18.537) arc (-74.87636:-105.12364:27.795);
\draw (33,-20) node [below] {$H$};
\draw [black] (41.678,-32.463) arc (-157.13162:-202.86838:15.731);
\draw (39.94,-26.35) node [left] {$G$};
\draw [black] (44.684,-20.141) arc (25.88675:-25.88675:14.221);
\draw (46.61,-26.35) node [right] {$G$};
\end{tikzpicture}
  \caption{}
  \label{fig:two-four-cycles}
\end{figure}

Lemma \ref{lem:cancellation} suggests an approach for cancelling matchings $M \in \PM(G \cup H) \setminus \PM(G, H)$ or which are not simple.
For a matching $M \in \PM(G \cup H)$ and a vertex $v$, define the \textit{induced re-matching at $v$} to be 
$M' \in \PM(G \cup H)$ which agrees with $M$ for all vertices except $v$, where it
is the induced matching at $v$. Now, if there is a $G$-pair at $v$ in $M'$,
then $c_M$ can be cancelled out by summing over matchings at $v$ and using \cref{lem:cancellation}.
If $M'$ has only $(G, H)$-pairs at $v$, observe that (1)~the re-matching operation
increases the magnitude of $c_M$, i.e. $c_{M'} \geq c_M$ and (2)~the re-matching is
guaranteed to be $v$-simple as well.
Therefore $c_M$ is not dominant, and it can be upper bounded by the $v$-simple $(G,H)$-matching
$M'$.
We would like to continue to the next vertex until the only remaining dominant matchings are between $G$ and $H$ and are simple.
Based on this argument, we initially conjectured that \cref{conj:planar-inner-product}
held for all $G, H$ not necessarily planar.

Unfortunately this strategy doesn't work as some matchings $M$ may be needed to cancel out
re-matchings for multiple distinct vertices.
These matchings may appear with nonzero coefficients.
The observation is that such matchings must have certain ``crossing" structure.
To explain this we specialize to the case where $G,H$ have max degree 2. In this case
\[c_M = n^{\cycles(M)}\prod_{v \in V}\left\{\begin{array}{cc}
    1 & d(v) \leq 1\\
    1 &  d(v) = 2\text{ and no $G$-pair at }v\\
    \frac{-2}{n} & d(v) = 2\text{ and $G$-pair at }v
\end{array}\right. \]

We group the $c_M$ based on an overall matching $M \in \PM(G, H)$. This can be seen as applying the 
cancellation trick in \cref{lem:cancellation} simultaneously to all vertices.

\begin{definition}
    For multigraphs $G, H$ let $V_4$ be the set of vertices with $\deg_G(v) = \deg_H(v) = 2$.
    For $M \in \PM(G, H)$ and $v \in V_4$ let $M^{\oplus v} \in \PM(G \cup H)$ be defined by re-matching $v$ into a $G$-pair and $H$-pair instead of two $(G,H)$-pairs.
    For $S \subseteq V_4$ let $M^{\oplus S}$ re-match all vertices in $S$.\footnote{Note that
    the re-matchings in this definition are not necessarily induced re-matchings.}
\end{definition}

\begin{lemma}\label{lem:degree-4}
    Let $G, H$ with max degree 2 and let $V_4$ be the set of degree-4 vertices in $G \cup H$.
    \[ \E[p_G p_H] = \prod_{v \in V} n^{\rising{-d(v)}} \sum_{M \in \PM(G, H)} n^{\cycles(M)}\left(\sum_{S \subseteq V_4} (-1)^{\abs{S}} \frac{n^{\cycles(M^{\oplus S})}}{n^{\cycles(M) + \abs{S}}}\right).\]
\end{lemma}
\begin{proof}
    Recall from \cref{lem:cM},
    \[\E[p_G p_H] = \prod_{v \in V} n^{\rising{-d(v)}} \sum_{M \in \PM(G \cup H)} c_M.\]
    Each $M \in \PM(G \cup H) \setminus \PM(G, H)$ has a coefficient of
    $c_M = n^{\cycles(M)}\left(\frac{-2}{n}\right)^{g_M}$.
    For each $G$-pair, say at $v$, there are two ways to rematch at $v$ to get 
    two $(G,H)$-pairs. Split the coefficient $\frac{-2}{n}$ between these two rematchings.
\end{proof}

We say that a matching $M$ is ``uncancelled'' if the inner summation over $V_4$ is nonzero.
Let us fix $G, H$ and an uncancelled term $M \in \PM(G, H)$ and assume for
the sake of exposition that the 
inner summation is uncancelled and of order $\Omega(1)$; this is the maximum possible
magnitude for the inner summation, as the next lemma shows.
\begin{lemma}
$\cycles(M^{\oplus S}) \leq \cycles(M) + \abs{S}$
\begin{proof}
    We have $\cycles(M^{\oplus v}) \leq \cycles(M) + 1$ since the only way
    to increase the number of cycles is if one cycle splits into two.
    The claim follows by induction.
\end{proof}
\end{lemma}

We show how the non-cancelling property in this case is 
due to topological properties of $G$ and $H$ (\cref{lem:characterization-of-dom}).

\begin{definition}
    $S \subseteq V_4$ is dominant
    if $\cycles(M^{\oplus S}) = \cycles(M) + \abs{S}.$
\end{definition}

\begin{definition}
    For each cycle $C$ in $M$, let $\gloop(C)$ be the set of vertices $v$ such that 
    \begin{enumerate}[(i)]
        \item $v$ is visited twice in $C$,
        \item restricting the matching collection $M$ to all vertices except $v$,
        the induced matching at $v$ has a $G$-pair.
    \end{enumerate}
    Let $\gloop(M) = \bigcup_{C \in M} \gloop(C)$.
\end{definition}
In other words, $M^{\oplus v}$ for vertices in $\gloop(C)$ will split $C$ into
two subcycles. 
% Or, the  re-matching at vertices in $\gloop(C)$ will have a $G$-pair.
\cref{fig:gloop-example} gives an example with $G = \{\{1,2\},\{1,2\},\{2,3\},\{3,4\}\}, H = \{\{2,3\},\{3,4\}\}$. 
The same cycle $C$ is drawn in two different ways.
In the left image, the dashed lines are the matching collections for each vertex.
In the right image, a vertex which is visited more than once by $C$ is drawn more than once. 
The two vertices that are visited more than once are 2 and 3, and only 2 has
an induced $G$-pair, so $\gloop(C) = \{2\}$.

\begin{figure}[h]
    \centering
  {\includegraphics[width=0.45\textwidth]{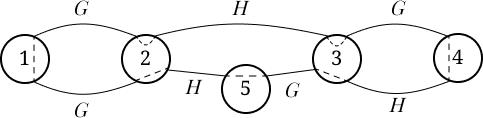}} \qquad
  {\includegraphics[width=0.3\textwidth]{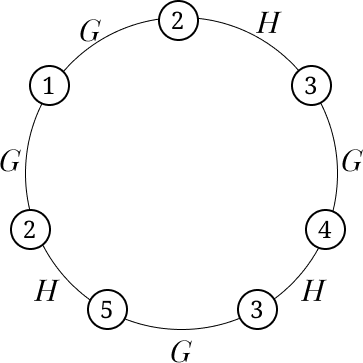}}
  \caption{$\gloop(C) = \{2\}.$}
  \label{fig:gloop-example}
\end{figure}

\begin{remark}
    If one defines $\textup{hh}(M)$ analogously using $H$, then $\textup{hh}(M) = \gloop(M)$.
\end{remark}

\begin{definition}
    $S \subseteq V_4$ is non-crossing if for each cycle $C$, $S$ is a non-crossing subset of $C$ drawn in a circle.
\end{definition}
In other words, the induced re-matching of all vertices in a non-crossing set $S$ in a cycle will subdivide the cycle. A picture is given in~\cref{fig:non-crossing}.

\begin{figure}[h]
    \centering
  {\includegraphics[width=0.25\textwidth]{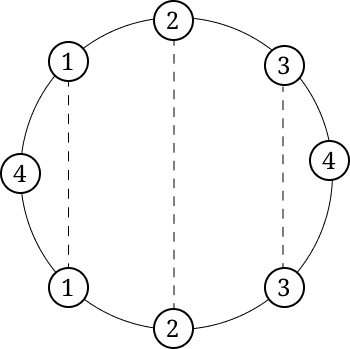}}\qquad
  {\includegraphics[width=0.6\textwidth]{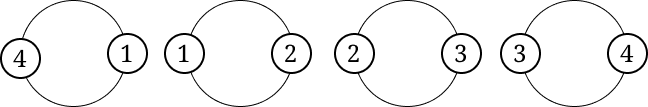}}
  \caption{In the left circle, $\{1,2,3\}$ is a non-crossing subset of $C$. The right four circles show the result
  of the induced rematching of $\{1,2,3\}$.}
  \label{fig:non-crossing}
\end{figure}

The key lemma is that the dominant terms are precisely non-crossing subsets of $\gloop(M)$:
\begin{lemma}
\label{lem:characterization-of-dom}
    $S$ is dominant if and only if $S$ is a non-crossing subset of $\gloop(M)$.
\end{lemma}

\begin{proposition}
    \label{lem:dom-subset}
    If $S$ is dominant then any subset $S' \subseteq S$ is also dominant.
\end{proposition}
\begin{proof}
    We have $\cycles(M^{\oplus v}) \leq \cycles(M) + 1$. If $S'$ is not dominant,
    meaning $\cycles(M^{\oplus S'}) < \cycles(M) + \abs{S'}$, then
    $\cycles(M^{\oplus S})$ cannot ``catch up'' to $\cycles(M) + \abs{S}$.
\end{proof}

\begin{proof}[Proof of~\cref{lem:characterization-of-dom}]
    First, observe that any non-crossing subset of $\gloop(M)$ is dominant.
    $v \in \gloop(M)$ ensures that $M^{\oplus v}$ splits the cycle containing $v$
    into two cycles.
    Because the set of vertices is non-crossing, further splits will create one
    new cycle each time.
    
    Now we show the converse. Let $S \subseteq V_4$ be dominant. There are three
    possibilities for $v \in V_4$: 
    \begin{enumerate}[(i)]
        \item $v$ is in two different cycles of $M$, 
        \item $v$ occurs twice in the same cycle and is in $\gloop(M)$, 
        \item $v$ occurs twice in the same cycle and is not in $\gloop(M)$.
    \end{enumerate}
    In the first case, $\cycles(M^{\oplus v}) = \cycles(M) - 1$ decreases. Therefore $\{v\}$ is not dominant, and by \cref{lem:dom-subset}, $v$ cannot be in $S$.
    In the third case, $\cycles(M^{\oplus v}) = \cycles(M)$. Again, $\{v\}$ is not dominant
    and $v$ cannot be in $S$. We deduce $S \subseteq \gloop(M)$.

    Next, we claim that a pair of crossing vertices $v, w \in \gloop(C)$ are not dominant.
    We have $\cycles(M^{\oplus \{v,w\}}) = \cycles(M)$ whereas a dominant term should
    increase the cycle count by 2.
    Therefore, by \cref{lem:dom-subset} we conclude that $S$ cannot contain any crossing
    pairs. This completes the proof of the lemma.
\end{proof}

Let $s_M$ be the leading coefficient,
\begin{definition}
    $\displaystyle s_M \defeq \sum_{\text{non-crossing } S \subseteq \gloop(M)} (-1)^{\abs{S}}$.
\end{definition}

One step towards \cref{conj:planar-inner-product} is to show that for planar graphs,
an uncancelled term is always upper bounded by a simple matching. A concrete, purely combinatorial
conjecture is the following,
\begin{conjecture}
    If $G, H$ are graphs, $G \cup H$ is planar, $M \in \PM(G, H)$ and $s_M \neq 0$,
    then there is a simple matching $M' \in \PM(G, H)$ with $\cycles(M') \geq \cycles(M).$
\end{conjecture}
It is easy to check that a simple $M \in \PM(G, H)$ is always
uncancelled since $\cycles(M^{\oplus S})<\cycles(M) + \abs{S}$ holds with strict inequality for all $S \neq \emptyset$.
If the above conjecture is true, terms with $s_M \neq 0$ will therefore be dominated
by simple terms.

\section{Polynomial Basis for the Boolean Setting}
\label{sec:boolean-case}

Let $H_n = \{-1,+1\}^n$. Letting $d_u \unif H_n$, let $\symbool$
be the set of polynomials $p$ in the $d_u$ which are
symmetric under simultaneous automorphism of the hypercube:
for any $\pi \in \Aut(H_n)$, 
\[p(d_1, \dots, d_u, \dots) = p(\pi d_1, \dots, \pi d_u, \dots).\]

% Since we are working in the boolean domain, the polynomials are multilinear.
$\Aut(H_n)$ is well-known to be the hyperoctahedral group.
\begin{fact}
    $\Aut(H_n)$ consists of permutations of the coordinates $[n]$
    and bitflips using any $z \in \{-1,+1\}^n$. Formally, $\Aut(H_n)$ is a semidirect product of $S_n$ and $\Z_2^n$.
\end{fact}

We give a nice basis for such functions, showing formulas that mirror the general
theme of routings and matchings in the underlying graph.

\begin{definition}[(Generalized inner product)]
    For $d_1, \dots, d_{2k}$, let
    \[ \langle d_1, \dots, d_{2k}\rangle = \sum_{i=1}^n d_{1,i} \cdots d_{2k,i}. \]
    % For $e = \{e_1, \dots, e_{2k}\} \subseteq [m]$ we abbreviate
    % \[ \langle d_e \rangle = \langle d_{e_1}, \dots, d_{e_{2k}}\rangle.\]
This is also denoted by the variable $x_{1,\dots, 2k}$.
\end{definition}

We say that a hypergraph is even if the size of every hyperedge is even.
Let $\deg_G(v)$ be the number of edges of $G$ containing $v$.
Given an even hypergraph $G$ on vertex set $[m]$, let
% \[m_G \defeq  \prod_{e \in E(G)} \langle d_e\rangle .\]
\[m_G =  \prod_{\{e_1, \dots, e_{2k}\} \in E(G)} x_{e_1, \dots, e_{2k}} 
% \langle d_{e_1}, \dots, d_{e_{2k}}\rangle
.\]
Note that edges are allowed to repeat.

The $m_G$ are not linearly independent. A basis is:
\begin{lemma}
    The set of $m_G$ such that: there is $\sigma : E(G) \to [n]$ such that for all vertices 
    $u \in V$ and edges $e, f \ni u$, $\sigma(e) \neq \sigma(f)$, is a basis for $\symbool$.
\end{lemma}
\begin{proof}
    Expand 
    \[m_G = \sum_{\sigma : E(G) \to [n]} \prod_{e = \{e_1, \dots, e_{2k} \in E(G)\}} d_{e_1, \sigma(e)} \cdots d_{e_{2k}, \sigma(e)}.\]
    If there is no such $\sigma$, then every term above has a square term $d_{ij}^2 = 1$.
    Therefore $m_G$ simplifies to a lower-degree polynomial, and it can be expressed
    in terms of other $m_G$.
    
    If there is a $\sigma$ for $G$, then $m_G$ contains a multilinear monomial with ``shape'' $G$,
    which is linearly independent from other $m_G$.
    More formally, to show linear independence, suppose $\sum_G c_G m_G = 0$ for some $c_G$ not all zero. Taking a 
    nonzero graph $G$ with maximum number of edges,
    precisely the coefficient $c_G$ appears on multilinear monomials with ``shape'' $G$,
    such as the monomial for $\sigma$,
    which is a contradiction.
\end{proof}
\begin{corollary}
    The set of $m_G$ such that $G$ has at most $n$ hyperedges is linearly independent. 
\end{corollary}

\begin{remark}
    The hyperedges are sets, so they don't contain repeats (and thus $G$ has no self-loops). If we did have an edge $e$ with a repeated vertex $i$ in $G$, we could delete two copies of $i$ from $e$ without affecting $m_G$ because we always have that $d_{ij}^2 = 1$.
\end{remark}

As before, we can run Gram-Schmidt to orthogonalize the $m_G$.
We will generalize matching collections to the Boolean case and use them to express
the resulting polynomials $p_G$. In the Boolean case it is also
useful to express $p_G$ and various calculations as a sum over certain functions
$\sigma: E(G) \to [n]$.

\begin{definition}%[(Matching collections on $G$)]
    Let $\calM_{bool}(G)$ be the set of partitions of $E(G)$.
\end{definition}

% A block of $M \in \widetilde{\calM}_{bool}(G)$ may not be a connected subgraph
% (unlike in the Gaussian/spherical cases, where the ``blocks'' 
% correspond to paths formed by $M$).

% \begin{definition}
%     Let $\calM_{bool}(G) \subseteq \widetilde{\calM}_{bool}(G)$ be $M$ such that the edges in each block
%     induce a connected subgraph.
% \end{definition}

\begin{definition}
    For $M \in \calM_{bool}(G)$ define the routed hypergraph $\route(M)$ by
    replacing each block $B$ by a single hyperedge containing $v \in V$ which are incident to an odd
    number of edges in $B$.
    
    Any block such that every $v \in V$ is incident to an even number of edges in $B$ is called a ``closed block''. 
    Closed blocks are deleted from $\route(M)$.
\end{definition}

\begin{definition}
    For $M \in \calM_{bool}(G)$ define the notation $\cycles(M)$ to be the number of closed blocks of the partition.
\end{definition}

\begin{definition}%[(Perfect matching collections on $G$)]
    Let $\mathcal{PM}_{bool}(G)$ be the set of partitions of $E(G)$ such that
    every block is closed. 
    % Define $\mathcal{PM}_{bool}(G)$ likewise.
\end{definition}

Denote the falling and rising factorial by 
\[x^{\underline{k}} \defeq x(x-1)\cdots (x-k+1), \qquad \qquad x^{\overline{k}} \defeq x(x+1)\cdots (x+k-1).\]

\begin{lemma}\label{lem:boolean-expectation}
    \[\E[m_G] = \sum_{\substack{\sigma : E(G) \to [n]\\
    \text{s.t. }\forall i. \; \sigma^{-1}(i)\text{ even}}}1 = \sum_{M \in \PM_{bool}(G)} n^{\underline{\cycles(M)}}. \]
\end{lemma}
\begin{proof}
    The first equality is obtained by expanding $m_G$ into a sum of over all
    $\sigma :E(G) \to [n]$, then using linearity of expectation.
    The second equality is obtained by casing on which values of
    $\sigma(e)$ are equal, which induces a partition of $E(G)$.
    We have that $\sigma$ contributes to the first sum
    if and only if all of the blocks of the induced partition are
    closed.
    Once the partition is fixed, there are $n^{\underline{\cycles(M)}}$ ways to
    choose distinct values for each cycle.
\end{proof}
% Say that $M \in \PM(G)$ irreducible if no closed block can be divided into two
% nontrivial closed blocks.
% \[\E[m_G] = \sum_{irreducible M \in \PM(G)} n^{\cycles(M)}.\]

For now we give only one definition of $p_G$. The definition in terms of matchings is more
complicated and is included in \cref{app:boolean-mobius}.

\begin{definition}[(Generic construction from \cref{prop:generic-construction})]
\label{def:truncate-boolean}
    \[p_G = \sum_{\substack{\sigma: E(G) \to [n]\\
    \text{s.t. }\forall e,f \ni u. \; \sigma(e) \neq \sigma(f)}} 
    \prod_{e = \{e_1, \dots, e_{2k}\} \in E(G)} d_{e_1, \sigma(e)} d_{e_2, \sigma(e)} \cdots d_{e_{2k}, \sigma(e)}.\]
\end{definition}

\begin{lemma}[(Automorphism-invariance)]
    $p_G \in \symbool$.
    \begin{proof}
        Neither of the two types of $H_n$ symmetries changes $p_G$.
        Coordinate permutation doesn't change $p_G$ because $\sigma$ doesn't
        depend on the names of the coordinates. Bitflips don't change $p_G$
        because every hyperedge is even (so flips cancel out).
    \end{proof}
\end{lemma}

\begin{corollary}
    $p_G$ equals the output of the degree-orthogonal Gram-Schmidt process on the $m_G$.
\end{corollary}

We can easily compute the inner product of $p_G$ and $p_H$ in the Boolean case. The idea is that $p_G$ and $p_H$ only contain terms where each vertex appears in each block at most once. When we multiply $p_G$ and $p_H$ together, these blocks may merge, giving us blocks where each vertex appears at most twice. If there is a block where a vertex appears only once, this block will have zero expected value, so the only terms which have nonzero expected value are the terms where in each block, each vertex either doesn't appear or appears twice, once from a $G$-edge and once from an $H$-edge. We now make this argument more precise.

\begin{definition}%[(Perfect matching collections on $G, H$)]
    Let $\mathcal{PM}_{bool}(G, H)$ be the set of partitions of $E(G) \cup E(H)$ such that
    for each vertex and each block, the number of $G$-edges containing the vertex
   equals the number of $H$-edges. 
    % Define $\PM_{bool}(G, H)$ likewise.
    
    We say that a partition $M \in \mathcal{PM}_{bool}(G, H)$ is simple if for each block, each vertex appears at most $2$ times.
\end{definition}
\begin{lemma}
    \label{lem:boolean-inner-product}
    \begin{align*}
        \E[p_Gp_H] &= \abs{\Sigma(G, H)} =  \sum_{\text{simple }M \in \mathcal{PM}_{bool}(G, H)} n^{\underline{\cycles(M)}}
    \end{align*}
    where $\Sigma(G, H)$ is the set of functions $\sigma : E(G \cup H)~\to~[n]$ such that
    \begin{enumerate}[(i)]
        \item For $e, f \in E(G)$ such that $e \cap f \neq \emptyset$, $\sigma(e) \neq \sigma(f).$ 
        \item For $e, f \in E(H)$ such that $e \cap f \neq \emptyset$, $\sigma(e) \neq \sigma(f).$ 
        \item For all $u, i$, the size of $\{u \in e \in E(G \cup H) : \sigma(e)  =i \}$ is even.
        Note that from conditions (i) and (ii) it must be size either 0 or 2.
    \end{enumerate}
\end{lemma}
% \[\E[p_Gp_H] = \sum_{irreducible simple M \in \PM_{bool}(G, H)} n^{\cycles(M)}.\]
\begin{proof}
    The first equality follows from expanding $p_G, p_H$ and using linearity of expectation.
    The second equality follows from looking at the partition induced by $\sigma$.
    The definition of $\Sigma(G, H)$ exactly checks that this partition is simple
    and in $\PM_{bool}(G ,H)$.
\end{proof}
\begin{corollary}
    \label{cor:boolean-variance}
    $n^{\underline{\abs{E(G)}}} \leq \E[p_G^2] \leq (2\abs{E(G)})^{2\abs{E(G)}}n^{\underline{\abs{E(G)}}}$.
\end{corollary}
\begin{proof}
    Each cycle in $M$ requires at least two edges, and hence the maximum magnitude is
    bounded by $n^{\underline{\abs{E(G)}}}$. Furthermore,
    this can be achieved by matching each edge with its duplicate.
    The number of partitions of a $k$-element set is at most $k^k$, which proves the upper bound.
\end{proof}
The inner product formula implies that all inner products are non-negative, so the Boolean case does not exhibit the ``negative inner product'' abnormality of the spherical case with the $K_5$ example.

\section{Inversion Formula for Approximate Orthogonality}
\label{sec:inversion-formula}

Consider the problem of Fourier inversion: given parameters $\widehat{f}(G)$ for different graphs $G$, find an orthogonally invariant function $f:(\R^n)^V \to \R$ such that
\[ \ip{f}{p_G} = \widehat{f}(G) .\] If the $p_G$ were completely orthogonal, then the function
\[ f = \displaystyle\sum_G \widehat{f}(G) \cdot \frac{p_G}{\E p_G^2}\]
is the unique $f$ in the span of $p_G$ for given $G$. In general, let $Q$ be the square matrix indexed by graphs $G$ with entries $Q[G, H] \defeq \ip{p_G}{p_H}$. Then $f$ is given by
\[f = \displaystyle\sum_G \parens{\sum_H Q^{-1}[G, H] \cdot \widehat{f}(H)} \cdot p_G \]
provided that $Q$ is invertible.

Because of approximate orthogonality, $Q$ is close to a diagonal matrix. Therefore $Q^{-1}$ is also close to a diagonal matrix. Formally we show

\begin{lemma}\label{lem:approximate-inversion}
Suppose we are in either the Gaussian, spherical, or Boolean setting.
Let finitely many nonzero $\widehat{f}(G) \in \R$ be given where $G$ is a graph of the
appropriate type for the setting, and assume that $\abs{E(G)} = o\parens{\frac{\log n}{\log \log n}}$ for all given $G$. For sufficiently large $n$, there is a unique $f$ satisfying $\ip{f}{p_G} = \widehat{f}(G)$, and $f$ equals
\[f =  \displaystyle\sum_H \parens{\widehat{f}(H) + o(1) \cdot \max_{G \leftrightarrow H}\abs{\widehat{f}(G)}} \cdot \frac{p_H}{\E p_H^2}.\]
\end{lemma}
\begin{proof}
Since the $p_G$ are orthogonal if $G \notlrarrow H$, the matrix $Q$ is block diagonal with blocks defined by $\lrarrow$.
The bound on the size of given $G$ implies that the dimension of each block is $n^{o(1)}$.

The diagonal terms are 
\[\E[p_G^2] = \begin{cases}
    n^{\abs{E(G)} + o(1)} & \text{Gaussian case (\cref{cor:gaussian-variance})}\\
    n^{-\abs{E(G)} + o(1)} & \text{Spherical case (\cref{cor:spherical-variance})}\\
    n^{\abs{E(G)} + o(1)} & \text{Boolean case (\cref{cor:boolean-variance})}
\end{cases}.\]

The off-diagonal terms with $G \leftrightarrow H$ are bounded by
\[\abs{\E[p_G p_H]} \leq \begin{cases}
    \displaystyle\max_{M \in \PM(G, H)}n^{\cycles(M) + o(1)} & \text{Gaussian case (\cref{lem:gaussian-inner-product})}\\
    \displaystyle n^{-2\abs{E(G)}}\max_{M \in \PM(G, H)}n^{\cycles(M) + o(1)} & \text{Spherical case (\cref{lem:generic-ip-ubound})}\\
    \displaystyle\max_{\text{simple }M \in \PM_{bool}(G, H)}n^{\cycles(M) + o(1)} & \text{Boolean case (\cref{lem:boolean-inner-product})}
\end{cases}.\]
When $G \neq H$, we claim that $\cycles(M)$ must be strictly less than $\abs{E(G)}$.
Any $M\in\PM(G, H)$ achieving $\abs{E(G)}$ cycles must pair up edges of $G$ and $H$, 
which shows the contrapositive.

Therefore the off-diagonal terms are smaller by a factor of $n^{1-o(1)}$ than the diagonal
term. 
Therefore $Q$ is invertible (for sufficiently large $n$) and 
\[Q^{-1}[G, H] = \begin{cases}
    \frac{1}{\E [p_G^2]} & G = H\\
    n^{o(1)-1} \cdot \frac{1}{\E [p_G^2]} & G \neq H
\end{cases}.\]
This proves the approximate Fourier inversion.
\end{proof}
\begin{remark}
    Using more careful counting, the assumption
    on $\abs{E(G)}$ can likely be improved to $\abs{E(G)} \leq n^\delta$ for
    some explicit $\delta > 0$.
\end{remark}

\bibliographystyle{alphaurl}
\bibliography{jones}

\appendix
\section{Tables of polynomials}
\label{sec:appendix}

The following table gives $p_G$ for the Gaussian case, $d_u \sim \calN(0, \Id_n)$, for all
connected graphs with up to 3 edges (up to isomorphism).

\noindent\begin{tabular}{|l|c|c|c|}
\hline
     Degree & $m_G$ & Picture of $G$ & $p_G$\\
    \hline
    0 & 1 & & 1\\
    \hline 
    1 & $x_{11}$ & \resizebox{!}{1em}{\begin{tikzpicture}
\tikzstyle{every node}+=[inner sep=0pt]
\draw [black] (35.3,-23.1) circle (3);
\draw [black] (32.62,-24.423) arc (324:36:2.25);
\end{tikzpicture}}  & $x_{11} - n$\\
    \hline
     1 & $x_{12}$ & \resizebox{!}{1em}{\begin{tikzpicture}
\tikzstyle{every node}+=[inner sep=0pt]
\draw [black] (35,-28.5) circle (3) node {1};
\draw [black] (48.6,-28.4) circle (3) node {3};
\draw [black] (45.6,-28.42) -- (38,-28.48);
\end{tikzpicture}} & $x_{12}$\\
     \hline
     2 & $x_{11}^2$ & 
     \resizebox{!}{1em}{\begin{tikzpicture}
\tikzstyle{every node}+=[inner sep=0pt]
\draw [black] (35.2,-25.1) circle (3);
\draw [black] (32.847,-26.943) arc (-24.19718:-312.19718:2.25);
\draw [black] (38.05,-24.2) arc (135.25384:-152.74616:2.25);
\end{tikzpicture}} & $x_{11}^2 - 2(n+2)x_{11} + (n+2)n$\\
     \hline
     2 & $x_{12}^2$ & 
     \resizebox{!}{0.5em}{\begin{tikzpicture}
\tikzstyle{every node}+=[inner sep=0pt]
\draw [black] (35,-28.5) circle (3) node {1};
%\draw [black] (41.9,-16.6) circle (3) node {2};
\draw [black] (48.6,-28.4) circle (3) node {3};
%\draw [black] (37,-26.4) -- (41,-20);
%\draw [black] (35, -24) -- (39, -17.5);
%\draw [black] (43.38,-19.21) -- (47.12,-25.79);
\draw [black] (45.6,-28.42) -- (38,-28.48);
\draw [black] (45.6,-26.42) -- (38,-26.48);
\end{tikzpicture}} & $x_{12}^2 - x_{11} - x_{22} + n$\\
     \hline
     2 & $x_{11}x_{12}$ & 
     \resizebox{!}{0.5em}{\begin{tikzpicture}
\tikzstyle{every node}+=[inner sep=0pt]
    \draw [black] (35.2,-24.6) circle (3);
    \draw [black] (46.7,-24.6) circle (3);
    \draw [black] (32.52,-25.923) arc (-36:-324:2.25);
    \draw [black] (38.2,-24.6) -- (43.7,-24.6);
\end{tikzpicture}} & $x_{11}x_{12} - (n+2)x_{12}$\\
    \hline
    2 &$x_{12}x_{23}$ &\resizebox{!}{0.5em}{\begin{tikzpicture}
\tikzstyle{every node}+=[inner sep=0pt]
\draw [black] (36.7,-34.3) circle (3);
\draw [black] (48.8,-34.3) circle (3);
\draw [black] (23.5,-34.3) circle (3);
\draw [black] (26.5, -34.3) -- (33.7, -34.3);
\draw [black] (45.8, -34.3) -- (39.7, -34.3);
\end{tikzpicture}} & $x_{12}x_{23} - x_{13}$\\
    \hline
    3 & $x_{11}^3$ & 
    \resizebox{!}{1em}{\begin{tikzpicture}
    \draw [black] (35.2,-24.6) circle (3);
    \draw [black] (32.52,-25.923) arc (-36:-324:2.25);
    \draw [black] (34.991,-21.619) arc (211.75098:-76.24902:2.25);
    \draw [black] (38.006,-25.628) arc (97.60282:-190.39718:2.25);
\end{tikzpicture}} & $x_{11}^3 - 3(n+4)x_{11}^2 + 3(n+4)(n+2)x_{11} - (n+4)(n+2)$ \\
    \hline
    3 & $x_{12}^3$ & 
    \resizebox{!}{0.5em}{\begin{tikzpicture}
\tikzstyle{every node}+=[inner sep=0pt]
\draw [black] (25.4,-26.6) circle (3) node {1};
\draw [black] (38.1,-26.6) circle (3) node {2};
\draw [black] (28.4, -26.6) -- (35.1, -26.6);
\draw [black] (35.734,-28.42) arc (-62.42395:-117.57605:8.606);
\draw [black] (28.108,-25.327) arc (107.91498:72.08502:11.84);
\end{tikzpicture}} & $x_{12}^3 - x_{12}$ \\
    \hline
    3 & $x_{11}x_{12}^2$ & 
    \resizebox{!}{1em}{\begin{tikzpicture}
\tikzstyle{every node}+=[inner sep=0pt]
\draw [black] (48.1,-24.2) circle (3);
\draw [black] (32.2,-24.2) circle (3);
\draw [black] (29.52,-25.523) arc (324:36:2.25);
\draw [black] (34.64,-22.469) arc (118.02282:61.97718:11.727);
\draw [black] (45.535,-25.743) arc (-65.57952:-114.42048:13.024);
\end{tikzpicture}} & \thead{$x_{11}x_{12}^2 - (n+4)x_{12}^2 - x_{11}x_{22} - x_{11}^2$ \\$+ (n+2)x_{22} + 2(n+2)x_{11} - n(n+2)$} \\
    \hline
    3 & $x_{11}^2x_{12}$ & 
    \resizebox{!}{1em}{\begin{tikzpicture}
\tikzstyle{every node}+=[inner sep=0pt]
    \draw [black] (48.1,-24.2) circle (3);
    \draw [black] (33.2,-24.2) circle (3);
    \draw [black] (30.326,-23.383) arc (281.86241:-6.13759:2.25);
    \draw [black] (32.628,-27.133) arc (16.69605:-271.30395:2.25);
    \draw [black] (36.2,-24.2) -- (45.1,-24.2);
\end{tikzpicture}} & $x_{11}^2x_{12} -2(n+4)x_{11}x_{12} + (n+4)(n+2)x_{12}$ \\
    \hline
    3 & $x_{11}x_{12}x_{22}$ & 
    \resizebox{!}{0.6em}{\begin{tikzpicture}
\tikzstyle{every node}+=[inner sep=0pt]
    \draw [black] (48.3,-24.2) circle (3);
    \draw [black] (32.2,-24.2) circle (3);
    \draw [black] (35.2,-24.2) -- (45.3,-24.2);
    \draw [black] (29.52,-25.523) arc (324:36:2.25);
    \draw [black] (50.98,-22.877) arc (144:-144:2.25);
\end{tikzpicture}} & $x_{11}x_{12}x_{22} - (n+2)(x_{11}x_{12} + x_{12}x_{22}) + (n+2)^2x_{12}$ \\

    \hline
    3 & $x_{12}^2x_{23}$ & 
    \resizebox{!}{0.6em}{\begin{tikzpicture}
\tikzstyle{every node}+=[inner sep=0pt]
    \draw [black] (48.1,-24.2) circle (3);
    \draw [black] (32.2,-24.2) circle (3);
    \draw [black] (64.3,-24.2) circle (3);
    \draw [black] (34.64,-22.469) arc (118.02282:61.97718:11.727);
    \draw [black] (45.535,-25.743) arc (-65.57952:-114.42048:13.024);
    \draw [black] (61.3,-24.2) -- (51.1,-24.2);
\end{tikzpicture}}
    & $x_{12}^2x_{23} - x_{11}x_{23} - x_{22}x_{23} - 2x_{12}x_{13}  + (n+2) x_{23} $\\
    \hline
    3 & $x_{11}x_{12}x_{23}$ & 
    \resizebox{!}{0.6em}{\begin{tikzpicture}
\tikzstyle{every node}+=[inner sep=0pt]
    \draw [black] (48.1,-24.2) circle (3);
    \draw [black] (32.2,-24.2) circle (3);
    \draw [black] (64.3,-24.2) circle (3);
    \draw [black] (35.2,-24.2) -- (45.1,-24.2);
    \draw [black] (61.3,-24.2) -- (51.1,-24.2);
    \draw [black] (29.52,-25.523) arc (-36:-324:2.25);
\end{tikzpicture}}
    & $x_{11}x_{12}x_{23} - (n+2)x_{12}x_{23} - x_{11}x_{13} + (n+2)x_{13}$\\
    \hline
    3 & $x_{12}x_{22}x_{23}$ & 
    \resizebox{!}{1em}{\begin{tikzpicture}
\tikzstyle{every node}+=[inner sep=0pt]
    \draw [black] (48.1,-24.2) circle (3);
    \draw [black] (32.2,-24.2) circle (3);
    \draw [black] (64.3,-24.2) circle (3);
    \draw [black] (35.2,-24.2) -- (45.1,-24.2);
    \draw [black] (61.3,-24.2) -- (51.1,-24.2);
\draw [black] (49.423,-26.88) arc (54:-234:2.25);\end{tikzpicture}}
    & $x_{12}x_{22}x_{23} - (n+4)x_{12}x_{23} - x_{13}x_{22} + (n+2)x_{13}$\\
    \hline
    3 & $x_{12}x_{23}x_{13}$ & \resizebox{!}{1em}{\begin{tikzpicture}
\tikzstyle{every node}+=[inner sep=0pt]
\draw [black] (35,-28.5) circle (3) node {1};
\draw [black] (41.9,-16.6) circle (3) node {2};
\draw [black] (48.6,-28.4) circle (3) node {3};
\draw [black] (36.5,-25.9) -- (40.4,-19.2);
\draw [black] (43.38,-19.21) -- (47.12,-25.79);
\draw [black] (45.6,-28.42) -- (38,-28.48);
\end{tikzpicture}} & $x_{12}x_{23}x_{13} - (x_{12}^2 + x_{13}^2 + x_{23}^2) + x_{11} + x_{22} + x_{33}  - n $\\
\hline
    3 & $x_{12}x_{13}x_{14}$ & \resizebox{!}{1em}{\begin{tikzpicture}
\tikzstyle{every node}+=[inner sep=0pt] 
\draw [black] (48.3,-33.9) circle (3);
\draw [black] (32.2,-24.2) circle (3);
\draw [black] (64.3,-24.2) circle (3);
\draw [black] (48.3,-54.6) circle (3);
\draw [black] (34.77,-25.75) -- (45.73,-32.35);
\draw [black] (61.73,-25.76) -- (50.87,-32.34);
\draw [black] (48.3,-51.6) -- (48.3,-36.9);
\end{tikzpicture}} & $x_{12}x_{13}x_{14} - (x_{12}x_{34} + x_{13}x_{24} + x_{14}x_{23})$\\
    \hline
    3 & $x_{12}x_{23}x_{34}$ &
    \resizebox{!}{0.5em}{\begin{tikzpicture}
\tikzstyle{every node}+=[inner sep=0pt] 
    \draw [black] (48.3,-24.2) circle (3);
    \draw [black] (32.2,-24.2) circle (3);
    \draw [black] (64.3,-24.2) circle (3);
    \draw [black] (15.9,-24.2) circle (3);
    \draw [black] (35.2,-24.2) -- (45.3,-24.2);
    \draw [black] (61.3,-24.2) -- (51.3,-24.2);
    \draw [black] (18.9,-24.2) -- (29.2,-24.2);
\end{tikzpicture}} 
    & $x_{12}x_{23}x_{34} - (x_{12}x_{24} + x_{13}x_{34}) + x_{14}$\\
 \hline
\end{tabular}

The following table gives $p_G$ for the spherical case, $d_u \unif S^{n-1}$, for all connected, loopless graphs with up to 4 edges (up to isomorphism).

\noindent\begin{tabular}{|l|c|c|c|}
\hline
     Degree & $m_G$ & Picture of $G$ & $p_G$\\
    \hline
    0 & 1 & & 1\\
    \hline
     1 & $x_{12}$ & \resizebox{!}{1em}{\begin{tikzpicture}
\tikzstyle{every node}+=[inner sep=0pt]
\draw [black] (35,-28.5) circle (3) node {1};
\draw [black] (48.6,-28.4) circle (3) node {3};
\draw [black] (45.6,-28.42) -- (38,-28.48);
\end{tikzpicture}} & $x_{12}$\\
     \hline
     2 & $x_{12}^2$ & 
     \resizebox{!}{0.5em}{\begin{tikzpicture}
\tikzstyle{every node}+=[inner sep=0pt]
\draw [black] (35,-28.5) circle (3) node {1};
%\draw [black] (41.9,-16.6) circle (3) node {2};
\draw [black] (48.6,-28.4) circle (3) node {3};
%\draw [black] (37,-26.4) -- (41,-20);
%\draw [black] (35, -24) -- (39, -17.5);
%\draw [black] (43.38,-19.21) -- (47.12,-25.79);
\draw [black] (45.6,-28.42) -- (38,-28.48);
\draw [black] (45.6,-26.42) -- (38,-26.48);
\end{tikzpicture}} & $x_{12}^2 - \frac{1}{n}$\\
    \hline
    2 &$x_{12}x_{23}$ &\resizebox{!}{0.5em}{\begin{tikzpicture}
\tikzstyle{every node}+=[inner sep=0pt]
\draw [black] (36.7,-34.3) circle (3);
\draw [black] (48.8,-34.3) circle (3);
\draw [black] (23.5,-34.3) circle (3);
\draw [black] (26.5, -34.3) -- (33.7, -34.3);
\draw [black] (45.8, -34.3) -- (39.7, -34.3);
\end{tikzpicture}} & $x_{12}x_{23} - \frac{1}{n}x_{13}$\\
    \hline
    3 & $x_{12}^3$ & 
    \resizebox{!}{0.5em}{\begin{tikzpicture}
\tikzstyle{every node}+=[inner sep=0pt]
\draw [black] (25.4,-26.6) circle (3) node {1};
\draw [black] (38.1,-26.6) circle (3) node {2};
\draw [black] (28.4, -26.6) -- (35.1, -26.6);
\draw [black] (35.734,-28.42) arc (-62.42395:-117.57605:8.606);
\draw [black] (28.108,-25.327) arc (107.91498:72.08502:11.84);
\end{tikzpicture}} & $x_{12}^3 - \frac{3}{n+2}x_{12}$ \\
    \hline
    3 & $x_{12}^2x_{23}$ & \resizebox{!}{0.6em}{\begin{tikzpicture}
\tikzstyle{every node}+=[inner sep=0pt]
    \draw [black] (48.1,-24.2) circle (3);
    \draw [black] (32.2,-24.2) circle (3);
    \draw [black] (64.3,-24.2) circle (3);
    \draw [black] (34.64,-22.469) arc (118.02282:61.97718:11.727);
    \draw [black] (45.535,-25.743) arc (-65.57952:-114.42048:13.024);
    \draw [black] (61.3,-24.2) -- (51.1,-24.2);
\end{tikzpicture}} & $x_{12}^2x_{23} -\frac{2}{n+2} x_{12}x_{13} - \frac{1}{n+2} x_{23} $\\
    \hline
    3 & $x_{12}x_{23}x_{13}$ & \resizebox{!}{1em}{\begin{tikzpicture}
\tikzstyle{every node}+=[inner sep=0pt]
\draw [black] (35,-28.5) circle (3) node {1};
\draw [black] (41.9,-16.6) circle (3) node {2};
\draw [black] (48.6,-28.4) circle (3) node {3};
\draw [black] (36.5,-25.9) -- (40.4,-19.2);
\draw [black] (43.38,-19.21) -- (47.12,-25.79);
\draw [black] (45.6,-28.42) -- (38,-28.48);
\end{tikzpicture}} & $x_{12}x_{23}x_{13} - \frac{1}{n}(x_{12}^2 + x_{13}^2 + x_{23}^2) + \frac{2}{n^2} $\\
\hline
    3 & $x_{12}x_{13}x_{14}$ & \resizebox{!}{1em}{\begin{tikzpicture}
\tikzstyle{every node}+=[inner sep=0pt] 
\draw [black] (48.3,-33.9) circle (3);
\draw [black] (32.2,-24.2) circle (3);
\draw [black] (64.3,-24.2) circle (3);
\draw [black] (48.3,-54.6) circle (3);
\draw [black] (34.77,-25.75) -- (45.73,-32.35);
\draw [black] (61.73,-25.76) -- (50.87,-32.34);
\draw [black] (48.3,-51.6) -- (48.3,-36.9);
\end{tikzpicture}} 
& $x_{12}x_{13}x_{14} - \frac{1}{n+2}(x_{12}x_{34} + x_{13}x_{24} + x_{14}x_{23})$\\
    \hline
    3 & $x_{12}x_{23}x_{34}$ &     \resizebox{!}{0.5em}{\begin{tikzpicture}
\tikzstyle{every node}+=[inner sep=0pt] 
    \draw [black] (48.3,-24.2) circle (3);
    \draw [black] (32.2,-24.2) circle (3);
    \draw [black] (64.3,-24.2) circle (3);
    \draw [black] (15.9,-24.2) circle (3);
    \draw [black] (35.2,-24.2) -- (45.3,-24.2);
    \draw [black] (61.3,-24.2) -- (51.3,-24.2);
    \draw [black] (18.9,-24.2) -- (29.2,-24.2);
\end{tikzpicture}} 
& $x_{12}x_{23}x_{34} - \frac{1}{n}(x_{12}x_{24} + x_{13}x_{34}) + \frac{1}{n^2}x_{14}$\\
    \hline
    4 & $x_{12}^4$ & 
    \resizebox{!}{0.8em}{\begin{tikzpicture}
\tikzstyle{every node}+=[inner sep=0pt]
    \draw [black] (43.8,-14.9) circle (3);
    \draw [black] (21,-14.9) circle (3);
    \draw [black] (23.958,-14.401) arc (98.14252:81.85748:59.605);
    \draw [black] (40.9,-15.665) arc (-77.42078:-102.57922:39.029);
    \draw [black] (23.137,-12.802) arc (128.68018:51.31982:14.822);
    \draw [black] (41.625,-16.959) arc (-52.26464:-127.73536:15.073);
\end{tikzpicture}} 
& $x_{12}^4 - \frac{6}{n+4}x_{12}^2 + \frac{3}{(n+4)(n+2)}$\\
    \hline
    4 & $x_{12}^3x_{23}$ & \resizebox{!}{0.5em}{\begin{tikzpicture}
\tikzstyle{every node}+=[inner sep=0pt]
\draw [black] (25.4,-26.6) circle (3) node {1};
\draw [black] (38.1,-26.6) circle (3) node {2};
\draw [black] (50.8,-26.6) circle (3) node {2};
\draw [black] (28.4, -26.6) -- (35.1, -26.6);
\draw [black] (41.1, -26.6) -- (47.8, -26.6);
\draw [black] (35.734,-28.42) arc (-62.42395:-117.57605:8.606);
\draw [black] (28.108,-25.327) arc (107.91498:72.08502:11.84);
\end{tikzpicture}} 
& $x_{12}^3x_{23} - \frac{3}{n+4} x_{12}^2x_{13} - \frac{3}{n+4}x_{12}x_{23} + \frac{3}{(n+4)(n+2)}x_{13}$\\
    \hline
    4 & $x_{12}^2x_{23}^2$ &\resizebox{!}{0.5em}{\begin{tikzpicture}
\tikzstyle{every node}+=[inner sep=0pt]
\draw [black] (25.4,-26.6) circle (3) node {1};
\draw [black] (38.1,-26.6) circle (3) node {2};
\draw [black] (51,-26.6) circle (3) node {3};
\draw [black] (35.734,-28.42) arc (-62.42395:-117.57605:8.606);
\draw [black] (28.108,-25.327) arc (107.91498:72.08502:11.84);
\draw [black] (48.58,-28.35) arc (-63.60504:-116.39496:9.066);
\draw [black] (40.834,-25.382) arc (107.18048:72.81952:12.581);
\end{tikzpicture}} & \thead{$x_{12}^2x_{23}^2 - \frac{4}{n+4}x_{12}x_{23}x_{13} - \frac{1}{n+4}(x_{12}^2 + x_{23}^2)$ \\ $ + \frac{2}{(n+4)(n+2)}x_{13}^2 + \frac{1}{(n+4)(n+2)} $}\\
    \hline
    4 & $x_{12}^2x_{23}x_{13}$ & \resizebox{!}{1.5em}{\begin{tikzpicture}
\tikzstyle{every node}+=[inner sep=0pt]
\draw [black] (34.6,-35.9) circle (3);
\draw [black] (46.3,-16) circle (3);
\draw [black] (24.2,-16) circle (3);
\draw [black] (32.022,-34.372) arc (-125.19145:-179.62423:18.975);
\draw [black] (33.21,-33.24) -- (25.59,-18.66);
\draw [black] (44.78,-18.59) -- (36.12,-33.31);
\draw [black] (27.2,-16) -- (43.3,-16);
\end{tikzpicture}}& \thead{$x_{12}^2x_{23}x_{13} - \frac{2}{n+2}(x_{12}x_{23}^2 + x_{12}x_{13}^2) - \frac{1}{n}x_{12}^3 - \frac{n-2}{(n+2)^2} x_{23}x_{13}$ \\ $ + \left(\frac{2}{(n+2)^2} + \frac{3}{n(n+2)}\right)x_{12}$}\\
    \hline
    4 & $x_{12}x_{23}x_{34}x_{14}$ & \resizebox{!}{2em}{\begin{tikzpicture}
\tikzstyle{every node}+=[inner sep=0pt]
\draw [black] (36.7,-34.3) circle (3);
\draw [black] (48.8,-34.3) circle (3);
\draw [black] (36.7,-20.8) circle (3);
\draw [black] (48.8,-20.8) circle (3);
\draw [black] (36.7,-23.8) -- (36.7,-31.3);
\draw [black] (39.7,-34.3) -- (45.8,-34.3);
\draw [black] (48.8,-31.3) -- (48.8,-23.8);
\draw [black] (45.8,-20.8) -- (39.7,-20.8);
%\draw [black] (38.8,-23.1) -- (46.8,-32.06);
%\draw [black] (38.7,-32.07) -- (46.8,-23.03);
\end{tikzpicture}} & \thead{$x_{12}x_{23}x_{34}x_{14} $\\$- \frac{1}{n}(x_{12}x_{23}x_{13} + x_{13}x_{34}x_{14} + x_{23}x_{34}x_{14} + x_{12}x_{24}x_{14})$\\ $+ \frac{1}{n^2}(x_{12}^2 + x_{13}^2 + x_{14}^2 + x_{23}^2 + x_{24}^2 + x_{34}^2) - \frac{3}{n^3} $}\\
\hline
    4 & $x_{12}x_{13}x_{14}x_{15}$ & 
\resizebox{!}{2em}{\begin{tikzpicture}
\tikzstyle{every node}+=[inner sep=0pt]
    \draw [black] (44.6,-34.4) circle (3);
\draw [black] (36.3,-24.2) circle (3);
\draw [black] (27.9,-34.4) circle (3);
\draw [black] (27.9,-13.9) circle (3);
\draw [black] (44.6,-13.9) circle (3);
\draw [black] (38.19,-26.53) -- (42.71,-32.07);
\fill [black] (42.71,-32.07) -- (42.59,-31.14) -- (41.81,-31.77);
\draw [black] (29.81,-32.08) -- (34.39,-26.52);
\fill [black] (34.39,-26.52) -- (33.5,-26.82) -- (34.27,-27.45);
\draw [black] (29.8,-16.22) -- (34.4,-21.88);
\fill [black] (34.4,-21.88) -- (34.29,-20.94) -- (33.51,-21.57);
\draw [black] (38.18,-21.86) -- (42.72,-16.24);
\end{tikzpicture}}
 & \thead{$x_{12}x_{13}x_{14}x_{15} $ \\ $- \frac{1}{n+4}(x_{12}x_{13}x_{45} + x_{12}x_{14}x_{35} + x_{12}x_{15}x_{34} + x_{13}x_{14}x_{25} $\\
 $+ x_{13}x_{15}x_{24} + x_{14}x_{15}x_{23}) $ \\ $+ \frac{1}{(n+4)(n+2)}(x_{23}{x_{45}} + x_{24}x_{35} + x_{25}x_{34})$}\\
    \hline 
    4 & $x_{12}x_{23}x_{34}x_{45}$ &  \resizebox{!}{0.5em}{\begin{tikzpicture}
\tikzstyle{every node}+=[inner sep=0pt] 
    \draw [black] (48.3,-24.2) circle (3);
    \draw [black] (32.2,-24.2) circle (3);
    \draw [black] (64.3,-24.2) circle (3);
    \draw [black] (80.3,-24.2) circle (3);
    \draw [black] (15.9,-24.2) circle (3);
    \draw [black] (35.2,-24.2) -- (45.3,-24.2);
    \draw [black] (61.3,-24.2) -- (51.3,-24.2);
    \draw [black] (67.3,-24.2) -- (77.3,-24.2);
    \draw [black] (18.9,-24.2) -- (29.2,-24.2);
\end{tikzpicture}}  & \thead{$x_{12}x_{23}x_{34}x_{45} -\frac{1}{n}(x_{13}x_{34}x_{45} + x_{12}x_{24}x_{45} + x_{12}x_{23}x_{35})$ \\ $ + \frac{1}{n^2}(x_{13}x_{35} + x_{14}x_{45} + x_{12}x_{25}) - \frac{1}{n^3}x_{15}$}\\
    \hline
    4 & $x_{12}^2x_{23}x_{34}$ & \resizebox{!}{0.6em}{\begin{tikzpicture}
\tikzstyle{every node}+=[inner sep=0pt]
    \draw [black] (48.1,-24.2) circle (3);
    \draw [black] (32.2,-24.2) circle (3);
    \draw [black] (64.3,-24.2) circle (3);
    \draw [black] (80.3,-24.2) circle (3);
    \draw [black] (34.64,-22.469) arc (118.02282:61.97718:11.727);
    \draw [black] (45.535,-25.743) arc (-65.57952:-114.42048:13.024);
    \draw [black] (61.3,-24.2) -- (51.1,-24.2);
    \draw [black] (67.3,-24.2) -- (77.1,-24.2);
\end{tikzpicture}} & \thead{$x_{12}^2x_{23}x_{34} - \frac{1}{n}x_{12}^2x_{24} - \frac{2}{n+2}x_{12}x_{13}x_{34} + \frac{2}{n(n+2)}x_{12}x_{14}$\\$ - \frac{1}{n+2}x_{23}x_{34}+ \frac{1}{n(n+2)}x_{24}$}\\
    \hline
    4 & $x_{12}x_{23}^2x_{34}$ & \resizebox{!}{0.6em}{\begin{tikzpicture}
\tikzstyle{every node}+=[inner sep=0pt]
    \draw [black] (48.1,-24.2) circle (3);
    \draw [black] (32.2,-24.2) circle (3);
    \draw [black] (16.2,-24.2) circle (3);
    \draw [black] (64.3,-24.2) circle (3);
    \draw [black] (34.64,-22.469) arc (118.02282:61.97718:11.727);
    \draw [black] (45.535,-25.743) arc (-65.57952:-114.42048:13.024);
    \draw [black] (61.3,-24.2) -- (51.1,-24.2);
    \draw [black] (19.2,-24.2) -- (29.2,-24.2);
\end{tikzpicture}} & \thead{$x_{12}x_{23}^2x_{34} - \frac{2}{n+2}(x_{12}x_{23}x_{24} + x_{13}x_{23}x_{34}) - \frac{n}{(n+2)^2}x_{12}x_{34}$ \\ $+ \frac{2}{(n+2)^2}(x_{14}x_{23} + x_{13}x_{24})$} \\
    \hline
    4 & $x_{12}x_{23}x_{34}x_{25}$ & \resizebox{!}{1.5em}{\begin{tikzpicture}
\tikzstyle{every node}+=[inner sep=0pt]
\draw [black] (48.8,-24.2) circle (3);
\draw [black] (36.3,-24.2) circle (3);
\draw [black] (23.9,-24.2) circle (3);
\draw [black] (36.3,-12.2) circle (3);
\draw [black] (61.4,-24.2) circle (3);
\draw [black] (39.3,-24.2) -- (45.8,-24.2);
\draw [black] (26.9,-24.2) -- (33.3,-24.2);
\draw [black] (36.3,-21.2) -- (36.3,-15.2);
\draw [black] (58.4,-24.2) -- (51.8,-24.2); 
\end{tikzpicture}} & \thead{$x_{12}x_{23}x_{34}x_{25} - \frac{1}{n+4}(x_{12}x_{34}x_{35} + x_{13}x_{25}x_{34} + x_{15}x_{23}x_{34})$ \\ $+\frac{1}{(n+4)(n+2)}(x_{12}x_{45} + x_{14}x_{25} + x_{15}x_{24})$}\\
    \hline
    4 & $x_{12}^2x_{13}x_{14}$ & 
    \resizebox{!}{2em}{\begin{tikzpicture}
\tikzstyle{every node}+=[inner sep=0pt]
    \draw [black] (47.7,-30.4) circle (3);
    \draw [black] (36.3,-24.2) circle (3);
    \draw [black] (24.9,-30.4) circle (3);
    \draw [black] (36.3,-12.2) circle (3);
    \draw [black] (38.94,-25.63) -- (45.06,-28.97);
    \draw [black] (27.54,-28.97) -- (33.66,-25.63);
    \draw [black] (34.822,-21.603) arc (-159.27957:-200.72043:9.62);
    \draw [black] (37.896,-14.724) arc (22.73965:-22.73965:8.992);
\end{tikzpicture}}
& \thead{$x_{12}^2x_{13}x_{14} $ \\ $- \frac{1}{n+4}(2x_{12}x_{13}x_{24} + 2x_{12}x_{14}x_{23} + x_{12}^2x_{34} + x_{13}x_{14}) $ \\ $+ \frac{1}{(n+4)(n+2)}(2x_{23}{x_{24}} + x_{34})$}\\
    \hline
    4 & $x_{12}x_{23}x_{13}x_{14}$ &
    \resizebox{!}{1.5em}{\begin{tikzpicture}
\tikzstyle{every node}+=[inner sep=0pt]
    \draw [black] (52,-24.2) circle (3);
    \draw [black] (36.3,-24.2) circle (3);
    \draw [black] (21,-32.4) circle (3);
    \draw [black] (21,-14.9) circle (3);
    \draw [black] (39.3,-24.2) -- (49,-24.2);
    \draw [black] (23.64,-30.98) -- (33.66,-25.62);
    \draw [black] (23.56,-16.46) -- (33.74,-22.64);
    \draw [black] (21,-29.4) -- (21,-17.9);
\end{tikzpicture}}
    & \thead{$x_{12}x_{23}x_{13}x_{14} - \frac{1}{n}(x_{12}^2x_{14} + x_{13}^2x_{14}) - \frac{1}{n+2}x_{14}x_{23}^2 $ \\$- \frac{1}{n+2}(x_{12}x_{23}x_{34} + x_{13}x_{23}x_{34}) $\\$+ \frac{2}{n(n+2)}x_{14} + \frac{2}{n(n+2)}(x_{12}x_{24} + x_{13}x_{34})$}\\
    \hline 
\end{tabular}

The following table gives $p_G$ for the Boolean case, $d_u \unif \{-1,+1\}^n$, for all
connected hypergraphs up to degree 6 in $d$ plus a few more cases.

\noindent\begin{tabular}{|l|c|c|}
    \hline
    Degree in $d$ & $m_G$ & $p_G$ \\
    \hline
    0 & $1$ & 1\\
    \hline
    2 & $x_{12}$ & $x_{12}$\\
    \hline
    4 & $x_{12}^2$ & $x_{12}^2 - n$\\
    \hline
    4 & $x_{12}x_{23}$ & $x_{12}x_{23} - x_{13}$\\
    \hline
    4 & $x_{1234}$ & $x_{1234}$\\
    \hline
    6 & $x_{12}^3$ & $x_{12}^3 - (3n-2)x_{12}$\\
    \hline
    6 & $x_{12}^2x_{23}$ & $x_{12}^2x_{23} -2x_{12}x_{13} - (n-2)x_{23} $\\
    \hline
    6 & $x_{12}x_{13}x_{23}$ & $x_{12}x_{13}x_{23} - (x_{12}^2 + x_{13}^2 + x_{23}^2) + 2n$\\
    \hline
    6 & $x_{12}x_{13}x_{14}$ & $x_{12}x_{13}x_{14} - (x_{12}x_{34}+ x_{13}x_{24} + x_{14}x_{23}) + 2x_{1234}$\\
    \hline
    6 & $x_{12}x_{23}x_{34}$ & $x_{12}x_{23}x_{34} - x_{13}x_{34} - x_{12}x_{24} + x_{14}$\\
    \hline
    6 & $x_{1234}x_{12}$ & $x_{1234}x_{12} - x_{34}$\\
    \hline
    6 & $x_{1234}x_{15}$ & $x_{1234}x_{15} - x_{2345}$\\
    \hline
    6 & $x_{123456}$ & $x_{123456}$\\
    \hline
    8 & $x_{12}^4$ & $x_{12}^4 - (6n - 8)x_{12}^2 + 3n^2 - 6n$\\
    \hline
    8 & $x_{12}x_{23}x_{34}x_{14}$ & \thead{$x_{12}x_{23}x_{34}x_{14} - (x_{12}x_{23}x_{13} + x_{13}x_{34}x_{14} + x_{23}x_{34}x_{14} + x_{12}x_{24}x_{14})$\\
    $+ (x_{12}^2 + x_{23}^2 + x_{34}^2 + x_{14}^2 + x_{13}^2 + x_{24}^2) - 3n$}\\
    \hline
\end{tabular}

\section{Proof of Isserlis' theorem}
\label{app:isserlis-theorem}

Isserlis' theorem is as follows:
\begin{theorem}
For any even $k$ and any fixed $d_1,\ldots,d_k$, 
\[
\E_{x \sim \calN(0,\Id_n)}\left[\prod_{i=1}^{k}{\langle{x,d_i}\rangle}\right] = \sum_{\text{perfect matchings } M}{\prod_{(i,j) \in M}{\langle{d_i,d_j}\rangle}}
\]
\end{theorem}
\begin{proof}
Let $d_{ia}$ be the $a$-th coordinate of $d_i$ and observe that 
\begin{align*}
\E_{x \sim \calN(0,\Id_n)}\left[\prod_{i=1}^{k}{\langle{x,d_i}\rangle}\right] &= \sum_{a_1,\ldots,a_k}{\left(\prod_{i=1}^{k}{d_{i{a_i}}}\right)\E\left[\prod_{i=1}^{k}{x_{a_i}}\right]} \\
&= \sum_{a_1,\ldots,a_k}{\left(\prod_{i=1}^{k}{d_{i{a_i}}}\right)\left|\{\text{perfect matchings }M: \forall (i,j) \in M, a_i = a_j\}\right|} \\
&= \sum_{\text{perfect matchings } M}{\sum_{a_1,\ldots,a_k: \forall (i,j) \in M, a_i = a_j}{\left(\prod_{i=1}^{k}{d_{i{a_i}}}\right)}} \\
&= \sum_{\text{perfect matchings } M}{\prod_{(i,j) \in M}{\langle{d_i,d_j}\rangle}}
\end{align*}
\end{proof}
\begin{theorem}
\label{thm:isserlis-proof}
For any even $k$, any fixed $d_1,\ldots,d_k$, and any $p \in \mathbb{Z} \cup \{0\}$,
\[
\E_{x \sim \calN(0,\Id_n)}\left[{\langle{x,x}\rangle}^{p}\prod_{i=1}^{k}{\langle{x,d_i}\rangle}\right] = \prod_{j=1}^{p}{(n + k + 2j - 2)}\sum_{\text{perfect matchings } M}{\prod_{(i,j) \in M}{\langle{d_i,d_j}\rangle}}
\]
\end{theorem}
\begin{proof}
To prove this, we can expand out $\prod_{i=1}^{k}{\langle{x,d_i}\rangle}$ in the same way and then use the following lemma to take care of the $\langle{x,x}\rangle$ factors.
\begin{lemma}
For any even powers $p_1,\ldots,p_n$,
\[
\E\left[\left(\sum_{j=1}^{n}{x_j^2}\right)\prod_{i=1}^{n}{x_i^{p_i}}\right] = \left(n + \sum_{j=1}^{n}{p_j}\right)\E\left[\prod_{i=1}^{n}{x_i^{p_i}}\right]
\]
\end{lemma}
\begin{proof}
Observe that for all $j \in [n]$, $\E\left[x_j^2\prod_{i=1}^{n}{x_i^{p_i}}\right] = (p_j+1)\E\left[\prod_{i=1}^{n}{x_i^{p_i}}\right]$, so 
\[
\E\left[\left(\sum_{j=1}^{n}{x_j^2}\right)\prod_{i=1}^{n}{x_i^{p_i}}\right] = \sum_{j=1}^{n}{(p_j + 1)}\E\left[\prod_{i=1}^{n}{x_i^{p_i}}\right] = \left(n + \sum_{j=1}^{n}{p_j}\right)\E\left[\prod_{i=1}^{n}{x_i^{p_i}}\right]
\]
\end{proof}
\end{proof}
\subsection{Analogue for the spherical case}
The spherical Isserlis' Theorem is as follows.
\begin{theorem}
For any even $k$ and any fixed $d_1,\ldots,d_k$, 
\[
\E_{x \sim S^{n-1}}\left[\prod_{i=1}^{k}{\langle{x,d_i}\rangle}\right] = \frac{1}{\prod_{i=1}^{\frac{k}{2}}{(n+2i-2)}}\sum_{\text{perfect matchings } M}{\prod_{(i,j) \in M}{\langle{d_i,d_j}\rangle}}
\]
\end{theorem}
\begin{proof}
To prove this, we use the same argument as we used to prove Isserlis' Theorem except that the moments of the monomials are different. As we show in Lemma \ref{lem:sphericalmonomialmoments}, for all $j \in \mathbb{N}$ and all even $p_1,\ldots,p_j$, letting $p_{total} = \sum_{i=1}^{j}{p_i}$,
\[
\E_{x \sim S^{n-1}}\left[\prod_{i=1}^{j}{x_i^{p_i}}\right] = \frac{\prod_{i=1}^{j}{(p_i!!)}}{\prod_{i=1}^{\frac{p_{total}}{2}}{(n+2i-2)}}
\]

Let $d_{ia}$ be the $a$-th coordinate of $d_i$ and observe that 
\begin{align*}
\E_{x \sim S^{n-1}}\left[\prod_{i=1}^{k}{\langle{x,d_i}\rangle}\right] &= \sum_{a_1,\ldots,a_k}{\left(\prod_{i=1}^{k}{d_{i{a_i}}}\right)\E_{x \sim S^{n-1}}\left[\prod_{i=1}^{k}{x_{a_i}}\right]} \\
&= \frac{1}{\prod_{i=1}^{\frac{k}{2}}{(n+2i-2)}}\sum_{a_1,\ldots,a_k}{\left(\prod_{i=1}^{k}{d_{i{a_i}}}\right)\left|\{\text{perfect matchings }M: \forall (i,j) \in M, a_i = a_j\}\right|} \\
&= \frac{1}{\prod_{i=1}^{\frac{k}{2}}{(n+2i-2)}}\sum_{\text{perfect matchings } M}{\sum_{a_1,\ldots,a_k: \forall (i,j) \in M, a_i = a_j}{\left(\prod_{i=1}^{k}{d_{i{a_i}}}\right)}} \\
&= \frac{1}{\prod_{i=1}^{\frac{k}{2}}{(n+2i-2)}}\sum_{\text{perfect matchings } M}{\prod_{(i,j) \in M}{\langle{d_i,d_j}\rangle}}
\end{align*}
\end{proof}
In the remainder of this appendix, we compute the moments of the monomials (see Lemma \ref{lem:sphericalmonomialmoments}).
\subsubsection{Preliminaries}
Before proving Lemma \ref{lem:sphericalmonomialmoments}, we need a few preliminaries.
\begin{definition}
For all even $n \in \mathbb{N} \cup \{0\}$, the double factorial of $n$ is $n!! = \prod_{i=1}^{\frac{n}{2}}{(2i-1)}$.
\end{definition}
\begin{lemma}\label{shiftingpowerslemma}
For all $a,b \in \mathbb{N} \cup \{0\}$ such that $a \geq 2$, 
\[
\int_{0}^{\frac{\pi}{2}}{sin(\Theta)^a{cos(\Theta)^{b}}d{\Theta}} = \frac{a-1}{b+1}\int_{0}^{\frac{\pi}{2}}{sin(\Theta)^{a-2}{cos(\Theta)^{b+2}}d{\Theta}}
\]
\end{lemma}
\begin{proof}
Observe that 
\[
\frac{d\left(-\frac{1}{b+1}sin(\Theta)^{a-1}{cos(\Theta)^{b+1}}\right)}{d{\Theta}} = sin(\Theta)^{a}{cos(\Theta)^{b}} - \frac{a-1}{b+1}sin(\Theta)^{a-2}{cos(\Theta)^{b+2}}
\]
which implies that 
\[
\int_{0}^{\frac{\pi}{2}}{sin(\Theta)^a{cos(\Theta)^{b}}d{\Theta}} = \frac{a-1}{b+1}\int_{0}^{\frac{\pi}{2}}{sin(\Theta)^{a-2}{cos(\Theta)^{b+2}}d{\Theta}}
\]
\end{proof}
\begin{lemma}\label{reducingpowerlemma}
For all integers $k \geq 2$, 
\[
\int_{0}^{\frac{\pi}{2}}{cos(\Theta)^{k}d{\Theta}} = \frac{k-1}{k}\int_{0}^{\frac{\pi}{2}}{cos(\Theta)^{k-2}d{\Theta}}
\] 
\end{lemma}
\begin{proof}
Observe that by Lemma \ref{shiftingpowerslemma}, 
\begin{align*}
\int_{0}^{\frac{\pi}{2}}{cos(\Theta)^{k}d{\Theta}} &= \int_{0}^{\frac{\pi}{2}}{cos(\Theta)^{k-2}d{\Theta}} - \int_{0}^{\frac{\pi}{2}}{sin(\Theta)^{2}cos(\Theta)^{k-2}d{\Theta}} \\
&= \int_{0}^{\frac{\pi}{2}}{cos(\Theta)^{k-2}d{\Theta}} - \frac{1}{k-1}\int_{0}^{\frac{\pi}{2}}{cos(\Theta)^{k}d{\Theta}}
\end{align*}
which implies that 
\[
\int_{0}^{\frac{\pi}{2}}{cos(\Theta)^{k}d{\Theta}} = \frac{k-1}{k}\int_{0}^{\frac{\pi}{2}}{cos(\Theta)^{k-2}d{\Theta}}
\] 
\end{proof}
\begin{corollary}\label{reducingpowercorollary}
For all $a,b \in \mathbb{N} \cup \{0\}$ such that $a \geq 2$, 
\[
\int_{0}^{\frac{\pi}{2}}{sin(\Theta)^a{cos(\Theta)^{b}}d{\Theta}} = \frac{a!!}{\prod_{i=1}^{\frac{a}{2}}{(b+2i)}}\int_{0}^{\frac{\pi}{2}}{cos(\Theta)^{b}d{\Theta}}
\]
\end{corollary}
\begin{proof}
Observe that by Lemma \ref{shiftingpowerslemma} and Lemma \ref{reducingpowerlemma}, 
\begin{align*}
\int_{0}^{\frac{\pi}{2}}{sin(\Theta)^a{cos(\Theta)^{b}}d{\Theta}} &= \prod_{i=1}^{\frac{a}{2}}{\frac{a-2i+1}{b+2i-1}}\int_{0}^{\frac{\pi}{2}}{cos(\Theta)^{a+b}d{\Theta}} \\
&= \left(\prod_{i=1}^{\frac{a}{2}}{\frac{(2i-1)}{b+a-2i+1}}\right)\left(\prod_{i=1}^{\frac{a}{2}}{\frac{a+b-2i+1}{a+b-2i+2}}\right)\int_{0}^{\frac{\pi}{2}}{cos(\Theta)^{b}d{\Theta}} \\
&= \frac{a!!}{\prod_{i=1}^{\frac{a}{2}}{(b+2i)}}\int_{0}^{\frac{\pi}{2}}{cos(\Theta)^{b}d{\Theta}}
\end{align*}
\end{proof}
\subsubsection{Computing the Moments of Monomials}
We can now compute the moments of the monomials. As a warm-up, we start with $x_1^k$.
\begin{lemma}
For all even $k$,
\[
\E_{x \unif S^{n-1}}[x_1^k] = \frac{k!!}{\prod_{i=1}^{\frac{k}{2}}{(n+2i-2)}}
\]
\end{lemma}
\begin{proof}%[Proof sketch]
Observe that if we let $A(n-1)$ be the surface area of the unit sphere in $\mathbb{R}^n$ then 
\[
\int_{x \in S^{n-1}}{x_1^k{dx}} = \int_{x_1=-1}^{1}{{x_1^k}\frac{1}{\sqrt{1 - x_1^2}}\left(\sqrt{1 - x_1^2}\right)^{n-2}A(n-2)dx_1}
\]
where the $\frac{1}{\sqrt{1 - x_1^2}}$ factor comes from the term from the slope of the curve $y = \sqrt{1 - x_1^2}$.
\[
\sqrt{1 + \left(\frac{d(\sqrt{1 - x_1^2})}{dx_1}\right)^2} = \sqrt{1 + \frac{x_1^2}{1 - x_1^2}} = \sqrt{\frac{1}{1 - x_1^2}}
\]
Plugging in $x_1 = sin(\Theta)$,
\[
\int_{x \in S^{n-1}}{x_1^k{dx}} = 2A(n-2)\int_{\Theta=0}^{\frac{\pi}{2}}{sin(\Theta)^k{cos(\Theta)^{n-2}}d{\Theta}}
\]
Following similar logic, 
\[
A(n-1) = \int_{x \in S^{n-1}}{1{dx}} = 2A(n-2)\int_{\Theta=0}^{\frac{\pi}{2}}{{cos(\Theta)^{n-2}}d{\Theta}}
\]
This implies that 
\[
E_{x \in S^{n-1}}[x_1^k] = \frac{\int_{x \in S^{n-1}}{x_1^k{dx}}}{A(n-1)} = \frac{\int_{0}^{\frac{\pi}{2}}{sin(\Theta)^k{cos(\Theta)^{n-2}}d{\Theta}}}{\int_{0}^{\frac{\pi}{2}}{{cos(\Theta)^{n-2}}d{\Theta}}}
\]
\end{proof}
\begin{lemma}\label{lem:sphericalmonomialmoments}
For all $j \in \mathbb{N}$ and all even $p_1,\ldots,p_j$, letting $p_{total} = \sum_{i=1}^{j}{p_i}$,
\[
\E_{x \sim S^{n-1}}\left[\prod_{i=1}^{j}{x_i^{p_i}}\right] = \frac{\prod_{i=1}^{j}{(p_i!!)}}{\prod_{i=1}^{\frac{p_{total}}{2}}{(n+2i-2)}}
\]
\end{lemma}
\begin{proof}%[Proof sketch]
Let $A(n-1)$ be the surface area of the unit sphere in $\mathbb{R}^n$, let $r_i = \sqrt{1 - \sum_{i'=1}^{i-1}{x_{i'}^2}}$, and observe that 
\[
\int_{x \in S^{n-1}}{\left(\prod_{i=1}^{j}{x_i^{p_i}}\right)dx} = 
\int_{x_1 = -1}^{1}{\int_{x_2 = -r_2}^{r_2}{\cdots \int_{x_j = -r_j}^{r_j}{\left(\prod_{i=1}^{j}{\frac{x_i^{p_i}{r_i}}{\sqrt{r_i^2 - x_i^2}}}\right){r_{j+1}^{n-j-1}}d{x_1}\ldots{d{x_j}}}}}
\]
where the $\frac{r_i}{\sqrt{r_i^2 - x_i^2}}$ term comes from the slope of the curve $y = \sqrt{r_i^2 - x_i^2}$. Setting $x_i = {r_i}sin(\Theta_i)$ and observing that for all $i \in [j]$, $r_{i+1} = {r_i}cos(\Theta_i)$, we have that 
\[
\int_{x \in S^{n-1}}{\left(\prod_{i=1}^{j}{x_i^{p_i}}\right)dx} = 
\int_{\Theta_1 = -\frac{\pi}{2}}^{\frac{\pi}{2}}{\cdots \int_{\Theta_j = -\frac{\pi}{2}}^{\frac{\pi}{2}}{\left(\prod_{i=1}^{j}{sin(\Theta_i)^{p_i}cos(\Theta_i)^{n - i - 1 + \left(\sum_{i'=i+1}^{j}{p_{i'}}\right)}}\right)d{\Theta_1}\ldots{d{\Theta_j}}}}
\]
Applying Corollary \ref{reducingpowercorollary} and Lemma \ref{reducingpowerlemma},
\begin{align*}
&\int_{\Theta_i = -\frac{\pi}{2}}^{\frac{\pi}{2}}{sin(\Theta_i)^{p_i}cos(\Theta_i)^{n - i - 1 + \left(\sum_{i'=i+1}^{j}{p_{i'}}\right)}d\Theta_i}\\
&= \frac{p_i!!}{\prod_{k = 1}^{\frac{p_i}{2}}{\left(n - i - 1 + 2k\right)}}\left(\prod_{k = 1}^{\sum_{i'=i+1}^{j}{\frac{p_{i'}}{2}}}{\frac{n - i + 2k - 2}{n - i + 2k - 1}}\right)\int_{\Theta_i = -\frac{\pi}{2}}^{\frac{\pi}{2}}{cos(\Theta_i)^{n - i - 1}d\Theta_i} \\
&= \frac{p_i!!\prod_{k = 1}^{\sum_{i'=i}^{j}{\frac{p_{i'}}{2}}}{\left(n - (i+1) - 1 + 2k\right)}}{\prod_{k = 1}^{\sum_{i'=i}^{j}{\frac{p_{i'}}{2}}}{\left(n - i - 1 + 2k\right)}}\int_{\Theta_i = -\frac{\pi}{2}}^{\frac{\pi}{2}}{cos(\Theta_i)^{n - i - 1}d\Theta_i}
\end{align*}
Putting these terms together, 
\[
\frac{\int_{x \in S^{n-1}}{\left(\prod_{i=1}^{j}{x_i^{p_i}}\right)dx}}{\int_{x \in S^{n-1}}{1dx}} = \frac{\prod_{i=1}^{j}{(p_i!!)}}{\prod_{k=1}^{\frac{p_{total}}{2}}{(n+2k-2)}}
\]
as needed. 
\end{proof}

\section{Inner product of \texorpdfstring{$K_5$}{K5} in the spherical case}
\label{app:k5}

Decompose the edges of $K_5$ into
\[G = \{1,2\}, \{2,3\},\{3,4\},\{4,5\}, \{1,5\}, \qquad \qquad H = \{1,3\}, \{1,4\}, \{2,4\}, \{2,5\}, \{3,5\}.
\]
as in \cref{fig:k5}. We compute that in the spherical setting,
\[\E[p_G p_H] = \frac{-8(n-1)(n-2)(n-4)}{n^8(n+2)^4}.\]

One way to compute $\E[p_G p_H]$ is to iteratively apply the following procedure as in \cref{ex:two-4-cycles}.
\begin{enumerate}
    \item Consider a vertex $v$ and partition the collections of matchings $M$ into cases based on the paths between the edges incident to $v$ and other edges incident to $v$ or edges incident to a previously considered vertex.
    \item For each case, find the factor given by summing over all of the matchings at $v$. For this, we assign a factor of $n$ to every cycle which contains $v$ and no previously considered vertex,
    a factor of $-2/n$ if the matching at $v$ is $G$-$G$, and a normalization factor
    of $\frac{1}{n(n+2)}.$
\end{enumerate}

We first consider vertex $5$. Let $a$ be the edge $\{5,1\}$, let $b$ be the edge $\{5,2\}$, let $c$ be the edge $\{5,3\}$, and let $d$ be the edge $\{5,4\}$. We now have the following cases
\begin{enumerate}
    \item There is a path from $a$ to $b$ and a path from $c$ to $d$. In this case, vertex $5$ gives a factor of 
    \[
    \frac{n^2}{n(n+2)} + \frac{n}{n(n+2)} - \frac{2n}{n^2(n+2)} = \frac{n^2 + n - 2}{n(n+2)} = \frac{n-1}{n}
    \]
    \item There is a path from $a$ to $c$ and a path from $b$ to $d$. This case is the same as the previous case except that $b$ and $c$ are swapped, so vertex $5$ will also give a factor of $\frac{n-1}{n}$ in this case.
    \item There is a path from $a$ to $d$ and a path from $b$ to $c$. In this case, by Lemma \ref{lem:cancellation}, everything cancels at vertex $5$.
\end{enumerate}
%We have a factor of $\frac{n^2}{n(n+2)} + \frac{n}{n(n+2)} - \frac{2n}{n^2(n+2)} = \frac{n^2 + n - 2}{n(n+2)} = \frac{n-1}{n}$ from vertex $5$. This is multiplied by a factor of $2$ because there are two symmetric cases.
Thus, it is sufficient to consider the first case and multiply the answer we obtain by $2\frac{n-1}{n}$. We now consider vertex $2$. Let $a'$ be the edge $\{2,1\}$, let $b' = b$ be the edge $\{2,5\}$, let $c'$ be the edge $\{2,4\}$, and let $d'$ be the edge $\{2,3\}$.

We will always have that $b' = b$. We have the following cases for how $a'$, $c'$, and $d'$ are connected to each other and/or $a$, $c$, and $d$.
\begin{enumerate}
\item If there is a path from $a$ to $a'$, there is a path from $c$ to $d$, and there is a path from $c'$ to $d'$ then vertex $2$ gives a factor of $\frac{n}{n(n+2)} + \frac{1}{n(n+2)} - \frac{2}{n^2(n+2)} = \frac{n^2 + n - 2}{n^2(n+2)} = \frac{n-1}{n^2}$.

There are two ways for this to happen. The path from $a$ to $a'$ must always consist of the edges $\{5,1\}$ and $\{1,2\}$. If the path from $c$ to $d$ is $\{5,3\}, \{3,4\}, \{4,5\}$ then the path from $c'$ to $d'$ is $\{2,4\}, \{4,1\}, \{1,3\}, \{3,2\}$. If the path from 
$c$ to $d$ is $\{5,3\}, \{3,1\}, \{1,4\}, \{4,5\}$ then the path from $c'$ to $d'$ is $\{2,4\}, \{4,3\}, \{3,2\}$. In either case, the vertices $1$, $3$, and $4$ give a factor of $\frac{4n}{(n^2(n+2))^3}$.
Thus, the total contribution from these cases is $\frac{8(n-1)}{n^7(n+2)^3} = \frac{8(n-1)(n+2)}{n^7(n+2)^4}$.
\item If there is a path from $a$ to $c'$, there is a path from $c$ to $d$, and there is a path from $a'$ to $d'$ then by Lemma \ref{lem:cancellation}, everything cancels at vertex $2$.
\item If there is a path from $a$ to $d'$, there is a path from $c$ to $d$, and there is a path from $a'$ to $c'$ then vertex $2$ gives a factor of $\frac{n-1}{n^2}$. 

There is only one way for this to happen. The path from $c$ to $d$ must consist of the edges $\{5,3\}, \{3,4\}, \{4,5\}$. The path from $a$ to $d'$ must consist of the edges $\{5,1\}, \{1,3\},\{3,2\}$. The path from $a'$ to $c'$ must consist of the edges $\{2,1\},\{1,4\},\{4,2\}$. In this case, the vertices $1$, 
$3$, and $4$ give a factor of $\frac{-2n^2}{(n^2(n+2))^3}$. Thus, the total contribution from this case is $\frac{-2n(n-1)}{n^7(n+2)^3} = \frac{-2n(n-1)(n+2)}{n^7(n+2)^4}$.
\item If there is a path from $a$ to $a'$, a path from $c$ to $c'$, and a path from $d$ to $d'$, vertex $2$ gives a factor of $\frac{1}{n(n+2)}$.

There are two ways for this to happen. The path from $a$ to $a'$ must always consist of the edges $\{5,1\}$ and $\{1,2\}$. If the path from $c$ to $c'$ is $\{5,3\}, \{3,4\}, \{4,2\}$ then the path from $d$ to $d'$ is $\{5,4\}, \{4,1\}, \{1,3\}, \{3,2\}$. In this case, the vertices $1$, $3$, and $4$ 
give a factor of $\frac{-2n^2}{(n^2(n+2))^3}$.  If the path from $c$ to $c'$ is $\{5,3\}, \{3,1\}, \{1,4\}, \{4,2\}$ then the path from $d$ to $d'$ is $\{5,4\}, \{4,3\}, \{3,2\}$. In this case, the vertices $1$, $3$, and $4$ give a factor of 
give a factor of $\frac{-8}{(n^2(n+2))^3}$.

Thus, the total contribution from these cases is $\frac{-2n^2 - 8}{n^7(n+2)^4}$.
\item  If there is a path from $a$ to $a'$, a path from $c$ to $d'$, and a path from $d$ to $c'$, vertex $2$ gives a factor of $\frac{1}{n(n+2)}$.

There are $7$ ways for this to happen depending on which path (if any) the cycle $\{1,3\}, \{3,4\}, \{4,1\}$ is incorporated into.

If the path from $c$ to $d'$ consists of the edges $\{5,3\}, \{3,2\}$, the path from $d$ to $c'$ consists of the edges $\{5,4\}, \{4,2\}$, and the path from $a$ to $a'$ consists of the edges $\{5,1\}, \{1,3\}, \{3,4\}, \{4,1\}, \{1,2\}$ or $\{5,1\}, \{1,3\}, \{3,4\}, \{4,1\}, \{1,2\}$ then the vertices $1$, $3$, and $4$ give a factor of $\frac{n^3}{(n^2(n+2))^3}$.

If the path from $a$ to $a'$ consists of the edges $\{5,1\}, \{1,2\}$ and the path from $d$ to $c'$ consists of the edges $\{5,4\}, \{4,2\}$, there are two choices for how the cycle is incorporated into the path from $c$ to $d'$. If the path from $c$ to $d'$ consists of the edges $\{5,3\}, \{3,1\}, \{1,4\}, \{4,3\}, \{3,2\}$ then the vertices $1$, $3$, and $4$ give a factor of $\frac{4n}{(n^2(n+2))^3}$. If the path from $d$ to $c'$ consists of the edges $\{5,3\}, \{3,4\}, \{4,1\}, \{1,3\}, \{3,2\}$ then the vertices $1$, $3$, and $4$ give a factor of $\frac{-2n^2}{(n^2(n+2))^3}$.

If the path from $a$ to $a'$ consists of the edges $\{5,1\}, \{1,2\}$ and the path from $c$ to $d'$ consists of the edges $\{5,3\}, \{3,2\}$, there are two choices for how the cycle is incorporated into the path from $d$ to $c'$. If the path from $d$ to $c'$ consists of the edges $\{5,4\}, \{4,1\}, \{1,3\}, \{3,4\}, \{4,2\}$ then the vertices $1$, $3$, and $4$ give a factor of $\frac{-2n^2}{(n^2(n+2))^3}$. If the path from $d$ to $c'$ consists of the edges $\{5,4\}, \{4,3\}, \{3,1\}, \{1,4\}, \{4,2\}$ then the vertices $1$, $3$, and $4$ give a factor of $\frac{4n}{(n^2(n+2))^3}$.

If the path from $a$ to $a'$ consists of the edges $\{5,1\}, \{1,2\}$, the path from $c$ to $d'$ consists of the edges $\{5,3\}, \{3,2\}$, the path from $d$ to $c'$ consists of the edges $\{5,4\}, \{4,2\}$ and we have an additional cycle $\{1,3\}, \{3,4\}, \{4,1\}$ then the vertices $1$, $3$, and $4$ give a factor of 
$\frac{-2n^3}{(n^2(n+2))^3}$. Note that this cancels with the two terms where the cycle is incorporated into the path from $a$ to $a'$.

Thus, the total contribution of these terms is 
\[
\frac{2n^3 + 2(-2n^2 + 4n) -2n^3}{n^7(n+2)^4} = \frac{-4n^2 + 8n}{n^7(n+2)^4}
\]
\item If there is a path from $a$ to $c'$, a path from $c$ to $a'$, and a path from $d$ to $d'$, vertex $2$ gives a factor of $\frac{-2}{n^2(n+2)}$.

There is only one way for this to happen. The path from $a$ to $c'$ must consist of the edges $\{5,1\}, \{1,4\},\{4,2\}$. The path from $c$ to $a'$ must consist of the edges $\{5,3\}, \{3,1\},\{1,2\}$. The path from $d$ to $d'$ must consist of the edges $\{5,4\}, \{4,3\}, \{3,2\}$. In this case, the vertices $1$, $3$, and $4$ give a factor of 
$\frac{4n}{(n^2(n+2))^3}$. 

Thus, the total contribution from this case is $\frac{-8}{n^7(n+2)^4}$.
\item If there is a path from $a$ to $c'$, a path from $c$ to $d'$, and a path from $d$ to $a'$, vertex $2$ gives a factor of $\frac{-2}{n^2(n+2)}$.

There are two ways for this to happen.  The path from $c$ to $d'$ must always consist of the edges $\{5,3\}$ and $\{3,2\}$. If the path from $a$ to $c'$ is $\{5,1\}, \{1,4\}, \{4,2\}$ then the path from $d$ to $a'$ is $\{5,4\}, \{4,3\}, \{3,1\}, \{1,2\}$. In this case, the vertices $1$, $3$, and $4$ 
give a factor of $\frac{-2n^2}{(n^2(n+2))^3}$.  If the path from $a$ to $c'$ is $\{5,1\}, \{1,3\}, \{3,4\}, \{4,2\}$ then the path from $d$ to $d'$ is $\{5,4\}, \{4,3\}, \{3,2\}$. In this case, the vertices $1$, $3$, and $4$ give a factor of 
give a factor of $\frac{n^3}{(n^2(n+2))^3}$.

Thus, the total contribution from these cases is $\frac{-2n^2 + 4n}{n^7(n+2)^4}$.
\item If there is a path from $a$ to $d'$, a path from $c$ to $a'$, and a path from $d$ to $c'$, vertex $2$ gives a factor of $\frac{1}{n(n+2)}$.

There are two ways for this to happen. The path from $d$ to $c'$ must always consist of the edges $\{5,4\}$ and $\{4,2\}$. If the path from $a$ to $d'$ is $\{5,1\}, \{1,3\}, \{3,2\}$ then the path from $c$ to $a'$ is $\{5,3\}, \{3,4\}, \{4,1\}, \{1,2\}$. In this case, the vertices $1$, $3$, and $4$ 
give a factor of $\frac{n^3}{(n^2(n+2))^3}$.  If the path from $a$ to $d'$ is $\{5,1\}, \{1,4\}, \{4,3\}, \{3,2\}$ then the path from $c$ to $a'$ is $\{5,3\}, \{3,1\}, \{1,2\}$. In this case, the vertices $1$, $3$, and $4$ give a factor of 
give a factor of $\frac{-2n^2}{(n^2(n+2))^3}$.

Thus, the total contribution from these cases is $\frac{n^3 - 2n^2}{n^7(n+2)^4}$.
\item If there is a path from $a$ to $d'$, a path from $c$ to $c'$, and a path from $d$ to $a'$, vertex $2$ gives a factor of $\frac{1}{n(n+2)}$.

There is only one way for this to happen. The path from $c$ to $c'$ must consist of the edges $\{5,3\}, \{3,4\},\{4,2\}$. The path from $a$ to $d'$ must consist of the edges $\{5,1\}, \{1,3\},\{3,2\}$. The path from $d$ to $a'$ must consist of the edges $\{5,4\}, \{4,1\}, \{1,2\}$. In this case, the vertices $1$, $3$, and $4$ give a factor of 
$\frac{n^3}{(n^2(n+2))^3}$. 

Thus, the total contribution from this case is $\frac{n^3}{n^7(n+2)^4}$.
\end{enumerate}
Adding everything together, we obtain
\begin{align*}
&\frac{8(n-1)(n+2) -2n(n-1)(n+2) + (-2n^2 - 8) + (-4n^2 + 8n) + (-8) + (-2n^2 + 4n) }{n^7(n+2)^4} \\
& \quad + \frac{(n^3 - 2n^2) + n^3}{n^7(n+2)^4}\\
&= \frac{-4n^2 + 24n - 32}{n^7(n+2)^4} = \frac{-4(n-2)(n-4)}{n^7(n+2)^4}
\end{align*}
Multiplying this by $2\frac{n-1}{n}$, our final answer is $\frac{-8(n-1)(n-2)(n-4)}{n^8(n+2)^4}$.

\section{Formulas using the partition poset}
\label{app:boolean-mobius}

Since $p_G$ is automorphism-invariant, it can be expressed in terms of the $m_G$ basis.
However, the coefficients on the $m_G$ are not that easy to work with.

\begin{definition}[(Routing definition)]
\label{def:routing-boolean}
    \[p_G =  \sum_{M \in \Lam^c_G} \mu(\emptyset, M) n^{\cycles(M)} m_{\route(M)}\]
where $\mu$ is the M\"obius function of the poset $\Lam^c_G$ (to be defined
in~\cref{def:connected-poset}).
\end{definition}
These coefficients can be computed by an inclusion-exclusion recurrence (which is in truth computing the M\"obius function
of a poset based on $G$, see \cite[Chapter 3]{StanleyBookVol1} for an overview
of poset combinatorics).
\begin{example}
Let $G$ have four parallel edges $\{s,t\}$.
Using~\cref{def:truncate-boolean},
\[p_G = \sum_{\text{injective }\sigma : [4] \to [n]} d_{s,\sigma(1)}d_{t,\sigma(1)}d_{s,\sigma(2)}d_{t,\sigma(2)}d_{s,\sigma(3)}d_{t,\sigma(3)}d_{s,\sigma(4)}d_{t,\sigma(4)}.\]
% Use the notation $x_{ij}$ or $x_{ijkl}$ for an inner product or generalized inner product of $d_i, d_j, d_k, d_l$. x_{st_1}x_{st_2}x_{st_3}x_{st_4}
The leading monomial is $\ip{d_s}{d_t}^4$. Subtract off terms where two edges are given the same label,
% \[ x_{st_1}x_{st_2}x_{} + x_{st_1}x_{st_3} + x_{st_1}x_{st_4} + x_{st_2}x_{st_3} + x_{st_2}x_{st_4}+ x_{st_3}x_{st_4}.\]
\[ \binom{4}{2}n\ip{d_s}{d_t}.\]
This puts a coefficient of $-2$ on terms with three equal labels and one unequal label. Add them back,
\[ \binom{4}{1}2\ip{d_s}{d_t}^2 .\]
The coefficient of terms with two pairs of two equal labels is $-1$. Add them back,
\[ 3 \cdot n^2.\]
Finally, the coefficient of the all-equal label is now 6. Subtract out
\[ 6n.\]

In total,
\[ p_G = \ip{d_s}{d_t}^4 - \binom{4}{2}n \ip{d_s}{d_t}^2 + \binom{4}{1} 2 \ip{d_s}{d_t}^2 + 3n^2 - 6n.\]
\end{example}

\begin{definition}
\label{def:poset}
    Let $\Lambda_G$ be the partition poset of $E(G)$: the elements are partitions of $E(G)$, and
    $M_1 \psdleq M_2$ if $M_1$ refines $M_2$.
\end{definition}
$\Lam_G$ has a unique minimal element (the partition into singletons, to be denoted by $\emptyset$) and a unique maximal element
(the partition with one block).
The example above corresponds to the standard partition poset of $\{1,2,3,4\}$~\cite[Example~3.10.4]{StanleyBookVol1}.

Observe that the poset refinement relation exactly captures how several coefficients on $m_G$
can contribute to the same coefficient on $d$. Stated formally, for $\sigma : E(G) \to [n]$ let $M(\sigma)$ denote the partition of $E(G)$ induced by $\sigma$. Then
\begin{fact}
    Given $\lam : \Lam_G \to \R$, 
    \[\sum_{M \in \Lam_G} \lam(M) n^{\cycles(M)} m_{\route(M)} =  \sum_{\sigma: E(G) \to [n]} \left(\sum_{M' \psdleq M(\sigma)} \lam(M') \right) \prod_{e = \{e_1, \dots, e_{2k}\} \in E(G)} d_{e_1, \sigma(e)}\cdots d_{e_{2k}, \sigma(e)}. \]
\end{fact}
Inverting the coefficients can be done by M\"obius inversion.
\begin{lemma}
\label{lem:mobius-inversion}
    Let $\kappa \subseteq \Lam_G$ be downward-closed. Let $\overline{\kappa} = \{\emptyset\} \cup (\Lam_G \setminus \kappa)$. Then
    \[\sum_{\substack{\sigma : E(G) \to [n]\\ \text{s.t. }M(\sigma) \in \kappa}} \prod_{e = \{e_1, \dots, e_{2k}\} \in E(G)}d_{e_1, \sigma(e)} \cdots d_{e_{2k}, \sigma(e)} = 
    \sum_{M \in \overline{\kappa}} \mu(\emptyset, M) n^{\cycles(M)} m_{\route(M)}\]
    where $\mu$ is the M\"obius function of $\overline{\kappa}$.
\end{lemma}

\begin{definition}
\label{def:connected-poset}
    Let $\Lam^c_G$ be the set of partitions of $E(G)$ such that either the partition is $\emptyset$ or
    there is a block and a vertex such that there are at least 2 edges of the block containing the vertex.
\end{definition}

% \begin{fact}[{(\cite[Example~3.10.4]{StanleyBookVol1})}]
%     The M\"obius function of the partition poset is $\mu(\pi, \tau)~=~(-1)^{n-1}(n-1)!$.
% \end{fact}

$\Lam_G^c = \overline{\kappa}$ where $\kappa$ is the (downward-closed) defining set of
partitions for $p_G$ in \cref{def:truncate-boolean}.
In summary from \cref{lem:mobius-inversion} we have,
\begin{lemma}
    \cref{def:routing-boolean} is equivalent to \cref{def:truncate-boolean}.
\end{lemma}

There is also a ``Boolean Isserlis theorem''.
The Boolean Isserlis theorem allows us to compute for fixed $d_{ij} \in \R^n$ and $v \unif \{-1,+1\}^n$,
\[ 
    \E_{v \unif \{-1,+1\}^n}[\langle v, d_{1,1},\dots , d_{1,\ell_1}\rangle \langle v, d_{2,1}, \dots, d_{2,\ell_2}\rangle 
    \cdots \langle v, d_{k,1}, \dots, d_{k,\ell_k}\rangle].
\]
Combining together the $d_{ij}$ component-wise, it suffices to compute
\[
    \E_{v \unif \{-1,+1\}^n}[\langle v, d_{1}\rangle \langle v, d_{2}\rangle 
    \cdots \langle v, d_{k}\rangle].
\]

Let $\Lam_{2k}$ be the partition poset for $2k$ elements and $\Lam^e_{2k}$ the subset
where each block has even size.
\begin{lemma}[(Boolean Isserlis theorem)]
    For fixed $d_{i} \in \R^n$ and $v \unif \{-1,+1\}^n$,
    \begin{align*}
     &\E_{v \unif \{-1,+1\}^n}[\langle v, d_{1}\rangle \langle v, d_{2}\rangle 
    \cdots \langle v, d_{2k}\rangle] = \sum_{\substack{\sigma : [2k] \to [n]\\\text{s.t. }\forall i.\; \abs{\sigma^{-1}(i)}\text{ even}}} 
    d_{1,\sigma(1)}\cdots d_{2k,\sigma(2k)}\\
    = & \sum_{M \in \Lambda^e_{2k}} \lam(M) \prod_{\{e_1, \dots, e_\ell\} \in M} \langle d_{e_1}, \dots, d_{e_\ell}\rangle
    \end{align*}
    where $\lam :\Lam_{2k}^e \to \R$ is defined by the recursion
    \begin{align*}
        \lam({\text{perfect matching}}) = 1,\\
        \sum_{M' \psdleq M} \lam(M') = 1.
    \end{align*}
\end{lemma}
\begin{proof}
The first equality is by expanding and applying linearity of expectation.
The second equality sums $\lam(M)$ in a way such that each coefficient on the relevant $d$ is 1 or 0.
The function $\lam : \Lam_{2k} \to \R$ needs to satisfy
\[\forall M \in \Lam_{2k}^e. \; \sum_{M' \psdleq M} \lam(M') = 1, \qquad \qquad \forall M \in \Lam_{2k} \setminus \Lam_{2k}^e. \; \sum_{M' \psdleq M} \lam(M') = 0.\]
There is a unique function, given by M\"obius inversion on $\Lam_{2k}$,
\[\lam(M) = \begin{cases}
\displaystyle\sum_{\substack{M' \psdleq M:\\M' \in \Lam_{2k}^e}} \mu(M', M) & M \in \Lam^e_{2k} \\
0 & M \not \in \Lam_{2k}^e 
\end{cases}\]
where $\mu(M', M)$ is the M\"obius function for $\Lam_{2k}$ given in~\cite[Example 3.10.4]{StanleyBookVol1}.
% \cnote{$\lam(M)$ should be explicitly computable 
% but I didn't do it...}
Equivalently, $\lam(M)$ must equal the recursion given in the lemma statement.
\end{proof}

Examples:
\begin{align*}
    \E [\ip{v}{d_1} \ip{v}{d_2}]  = & \ip{d_1}{d_2}\\
\E [\ip{v}{d_1} \ip{v}{d_2}\ip{v}{d_3}\ip{v}{d_4}] =& \ip{d_1}{d_2}\ip{d_3}{d_4} + \ip{d_1}{d_3}\ip{d_2}{d_4} + \ip{d_1}{d_4}\ip{d_2}{d_3} \\
    &- 2\langle d_1, d_2,d_3, d_4\rangle\\
    \E[\prod_{i=1}^6 \ip{v}{d_i}] =& \ip{d_1}{d_2}\ip{d_3}{d_4}\ip{d_5}{d_6} + \text{15 terms of type (2,2,2)}\\
    &- 2\ip{d_1}{d_2}\langle d_3, d_4, d_5, d_6 \rangle  + \binom{6}{2}\text{ terms of type (2,4)}\\
    &+ 16\langle d_1, d_2, d_3, d_4, d_5, d_6 \rangle \\
\E[\prod_{i=1}^8 \ip{v}{d_i}] =& \ip{d_1}{d_2}\ip{d_3}{d_4}\ip{d_5}{d_6}\ip{d_7}{d_8} + \text{105 terms of type (2,2,2,2)}\\
    &- 2\ip{d_1}{d_2}\ip{d_3}{d_4}\langle d_5, d_6, d_7, d_8\rangle + 3\binom{8}{4}\text{ terms of type (2,2,4)}\\
    &+ 16\ip{d_1}{d_2}\langle d_3, d_4, d_5, d_6, d_7, d_8\rangle + \binom{8}{2}\text{ terms of type (2,6)}\\
    &+ 4\langle d_1, d_2, d_3, d_4\rangle \langle d_5, d_6, d_7, d_8\rangle + \frac{1}{2}\binom{8}{4}\text{ terms of type (4,4)}\\
    &+ 8\langle d_1, d_2, d_3, d_4, d_5, d_6, d_7, d_8  \rangle
\end{align*}

\end{document}